\documentclass[12pt]{amsart}
\usepackage{newtxtext,newtxmath,mathrsfs}
\usepackage{anysize}
\usepackage{amsmath}
\usepackage{amsthm}
\usepackage{amsmath,amscd}
\usepackage{stmaryrd}
\usepackage{wasysym}
\usepackage{mathrsfs}
\usepackage{hyperref}
\usepackage{dsfont}
\usepackage{soul}
\hypersetup{colorlinks=true,linkcolor=blue,citecolor=black}
\usepackage{color}
\usepackage{float}
\usepackage{bm}
\usepackage{mathtools}
\usepackage{comment}
\usepackage{thmtools}
\usepackage{enumerate}
\usepackage{subcaption}
\usepackage{setspace}
\newcommand\mtop{1in}
\newcommand\mbottom{1in}
\newcommand\mleft{1in}
\newcommand\mright{1in}
\usepackage[top = \mtop, bottom = \mbottom, left = \mleft, right=\mright]{geometry}

\newtheorem{Theorem}{Theorem}[section]
\newtheorem{Remark}{Remark}
\newtheorem*{Hyp}{Hypothesis}
\newtheorem{Prop}[Theorem]{Proposition}
\newtheorem{Claim}{Claim}
\newtheorem{Lemma}[Theorem]{Lemma}

\theoremstyle{definition}
\newtheorem{Definition}{Definition}

\newcommand{\R}{\mathbb{R}}
\newcommand{\Z}{\mathbb{Z}}

\newcommand{\E}{\mathbb{E}}
\newcommand{\EE}{\mathbf{E}}
\usepackage{scalerel,stackengine}
\stackMath
\newcommand\reallywidehat[1]{%
    \savestack{\tmpbox}{\stretchto{%
            \scaleto{%
                \scalerel*[\widthof{\ensuremath{#1}}]{\kern-.6pt\bigwedge\kern-.6pt}%
                {\rule[-\textheight/2]{1ex}{\textheight}}
            }{\textheight}%
        }{0.5ex}}%
    \stackon[1pt]{#1}{\tmpbox}%
}

\newcommand{\eps}{\epsilon}

\title{Nontrivial upper bound for chemical distance on the planar random cluster model}

\author{Lily Reeves}
\email{lreeves@caltech.edu}
\thanks{The research of L.\ R.\ is supported by NSF grant DMS-2303316}

\begin{document}
\maketitle

\begin{abstract}
    We extend the upper bounds derived for the horizontal and radial chemical distance for $2d$ Bernoulli percolation in \cite{DHS21, SR20} to the planar random cluster model with cluster weight $1\le q\le 4$. Along the way, we provide a complete proof of the strong arm separation lemma for the random cluster model.
\end{abstract}

\section{Introduction}

\subsection{Random cluster model and chemical distance}

The random cluster model is a well-known dependent percolation model that generalizes classical models such as the Bernoulli percolation and Ising and Potts models. It was first introduced by Fortuin and Kasteleyn in 1972 \cite{FK72}. Let $G = (V,E)$ be a graph and $\omega$ be a percolation configuration on $G$ where each edge is colored open or closed. The random cluster measure is tuned by two parameters $p\in (0,1)$, the edge weight, and $q>0$, the cluster weight. Then, the random cluster measure is proportional to
\begin{equation*}
   p^{\text{\# open edges}} (1-p)^{\text{\# closed edges}} q^{\text{\# open clusters}}.
\end{equation*}
One can immediately see that when $q = 1$, the random cluster measure coincides with the Bernoulli percolation measure. For integer $q\ge 2$, the random cluster model corresponds to Ising and Potts models through the Edwards-Sokal coupling \cite{ES88}.

The random cluster model undergoes a phase transition at $p_c(q) = \sqrt{q}/(1+\sqrt{q})$ in the sense that for $p>p_c$, the probability the origin is contained in an infinite open cluster is positive while the probability is $0$ for $p< p_c$ \cite{BD12}. Moreover, 
the phase transition is continuous (i.e. the probability the origin is contained in an infinite open cluster is $0$ at criticality) for $q \in [1,4]$ \cite{DST17} and discontinuous for $q>4$ \cite{DGHMT16, RS20}.

The object of interest in this article is the \emph{chemical distance}. For two subsets $A$ and $B$ of $V$ and a percolation configuration viewed as a subgraph of $G$, the chemical distance is the graph distance between $A$ and $B$. Specifically, consider the box $B(n) = [-n,n]^2$ on $\Z^2$ and let $\mathcal H_n$ be the event that there exists a horizontal open crossing between the left and right sides of the box, we denote by $S_n$ the length of the shortest such crossing which we call \emph{the} chemical distance from now on.

Chemical distances on planar Bernoulli percolation have received attention from physicists and mathematicians alike. In both the subcritical and supercritical regimes, the chemical distance is known to behave linearly \cite{GM90,AP96}. In the critical phase, while physics literature \cite{EK85, Grassberger99, HN84, HTWB85, HHS84, HS88} generally assumes the existence of an exponent $s>0$ such that
\begin{equation*}
    E[S_n\mid \mathcal H_n] \sim n^s,
\end{equation*}
there is no widely accepted conjecture on the value of $s$, nor is there a precise interpretation of ``$\sim$''. 

The present known lower bound can be derived from the work of Aizenman and Burchard \cite{AB99}: there is $\eta > 0$ such that, with high probability
\begin{equation*}
    E[S_n \mid \mathcal H_n] \ge n^{1+\eta}.
\end{equation*}
This bound applies to a general family of random curves satisfying certain conditions. This includes shortest open connections in the random cluster model. We remark further on this lower bound in Section \ref{subsec:lower-bound}.

In \cite{KZ93}, Kesten and Zhang note that the shortest horizontal open crossing can be compared to the \emph{lowest crossing} $\ell_n$, which, by combinatorial arguments, consists of ``\emph{three-arm points}'' and has expected size
\begin{equation}\label{eq:AB-lower-bd}
    E[\# \ell_n \mid \mathcal H_n] \le Cn^2 P(A_3(n)),
\end{equation}
see \cite{MZ05}. Here $A_3(n)$ is the event that there are two disjoint open and one dual-closed paths from the origin to distance $n$. In \cite{DHS17, DHS21}, Damron, Hanson, and Sosoe improve the upper bound by a factor of $n^{-\delta}$ for some $\delta >0$, thus obtaining the present known upper bound:
\begin{equation}
    \label{eq:bernoulli}
    E[S_n \mid \mathcal H_n] \le Cn^{2-\delta} P(A_3(n)).
\end{equation}

In \cite{SR20}, Sosoe and the author obtain the same upper bound for the \emph{radial} chemical distance, which measures the expected length of the shortest open crossing from the origin to the boundary of the box $B(n)$ conditional on the existence of such a crossing. Although there is no lowest crossing to compare to in the radial case, the construction of a path consisting of three-arm points serves as the foundation for the improvement.

When $q\in [1,4]$, the random cluster model exhibits a continuous phase transition \cite{DST17} as well as enjoys positive association (FKG inequality). These facts combined with recent development in RSW-type quad-crossing probabilities \cite{DMT20} allows us to pursue an upper bound in the form of \eqref{eq:bernoulli}:

\begin{Theorem} \label{thm:main}
    Fix $1\le q\le 4$, $p=p_c(q)$ and let $\E$ denote the expectation with respect to the random cluster measure $\phi_{p_c, q, B(n)}^\xi$. For any boundary condition $\xi$, there is a $\delta>0$ and a constant $C>0$ independent of $n$ such that
    \begin{equation}
        \label{eq:main}
        \E [S_n \mid \mathcal H_n] \le Cn^{2-\delta} \phi_{p_c, q, B(n)}^\xi (A_{3}(n)).
    \end{equation}
\end{Theorem}

\subsection{Organization}

In Section \ref{subsec:notation}, we summarize the notations we use in this paper. In Section \ref{subsec:properties}, we list a few results and tools for the random cluster model we utilize. 

One of the persistent tools used in all of the above constructions is the so-called ``gluing construction''. In 2d critical Bernoulli percolation, this classical construction is realized by RSW estimates and generalized FKG inequality. The later is not known for the random cluster model. Therefore, we provide a detailed alternate argument for gluing constructions in Section \ref{sec:gluing}.

The general strategy to prove Theorem \ref{thm:main} aligns with \cite{DHS21}, which we outline in Section \ref{sec:outline} to provide context. We aim to point to the similarities and highlight the differences between the proofs for the two models to ensure readability while minimizing the amount of repetition. In Section \ref{sec:cond-decor} and \ref{sec:ext-sep}, we provide the proofs of a large deviation bound conditional on a three-arm event and the random-cluster analogue of the strong arm separation lemma. Both proofs involve strategic applications of the domain Markov property to circumvent the lack of independence.

We can extend the result of the main theorem to the radial chemical distance, following the approach in \cite{SR20} for Bernoulli percolation. Since most of the arguments in \cite{SR20} rely solely on independence and gluing constructions, they extend to the random cluster model when substituted with the domain Markov property and gluing constructions detailed in Section \ref{sec:gluing}. The remaining challenge is to find a way to bound the probability of a specific event without the use of Reimer's inequality. Such a method will be detailed in Section \ref{sec:no-reimer}.

\subsection{Notations} \label{subsec:notation}

In this paper, we consider the random cluster model on the square lattice $(\Z^2, \mathcal E)$, that is a graph with vertex set $\Z^2$ and edge set $\mathcal E$ consisting of edges between all pairs of nearest-neighbor vertices. We often work with the random cluster model on a discrete subdomain of $\Z^2$. A finite subdomain $\mathcal D = (V, E)$ is defined by the (finite) edge set $E$ and the vertex set $V$ of all endpoints of the edges in $E$. Its boundary $\partial \mathcal D$ consists of the vertices in the topological boundary of $\mathcal D$.

A percolation configuration $\omega = (\omega_e)_{e\in E}$ on a domain $\mathcal D = (V,E)$ is an element in the state space $\Omega = \{0,1\}^E$ which assigns a status to each edge $e\in E$. An edge $e$ is said to be open in $\omega$ if $\omega_e = 1$ and closed otherwise.

A boundary condition $\xi$ on $\mathcal D$ is a partition of $\partial \mathcal D$. All vertices in the same class of the partition are \emph{wired together} and count towards the same connected component when defining the probability measure. In the \emph{free} boundary condition, denoted by $0$ in the superscript, no two vertices on the boundary are identified with each other.

\begin{Definition}
    Let $\mathcal D = (V,E)$ be a subdomain of $\Z^2$. For an edge weight parameter $p\in [0,1]$ and a cluster weight parameter $q>0$, the random cluster measure on $\mathcal D$ with boundary condition $\xi$ is defined by
    \begin{equation*}
        \phi_{p, q, \mathcal D}^\xi(\omega) = \frac1{Z_{p, q, \mathcal D}^\xi} p^{o(\omega)}(1-p)^{c(\omega)} q^{k(\omega^\xi)} 
    \end{equation*}
    where $o(\omega)$ is the number of open edges in $\omega$, $c(\omega)= |E|-o(\omega)$ is the number of closed edges in $\omega$, $k(\omega^\xi)$ is the number of connected components of $\omega$ with consideration of the boundary condition $\xi$, and the partition function is defined by
    \begin{equation*}
        Z_{p, q, \mathcal D}^\xi = \sum_{\omega \in \{0,1\}^E}  p^{o(\omega)}(1-p)^{c(\omega)} q^{k(\omega^\xi)}.
    \end{equation*}
\end{Definition}

For the rest of this paper, we fix the cluster weight $q\in [1,4]$ and edge weight $p = p_c(q) = \sqrt{q}/(1+\sqrt{q}) $ and we drop them from the notation. 

\subsubsection{Arm events and arm exponents}

To set up for the so-called arm events, we first introduce a duality. The dual lattice of the square lattice is written as $((\Z^2)^*, \mathcal E^*)$ where $(\Z^2)^* = \Z^2+(1/2,1/2)$ and $\mathcal E^*$ its nearest-neighbor edges. For each edge $e \in \mathcal E$, $e^*\in \mathcal E^*$ is the dual edge that shares the same midpoint as $e$. Given $\omega \in \Omega$, we obtain $\omega^* \in \Omega^* = \{0,1\}^{\mathcal E^*}$ by the relation $\omega_e = \omega^*_{e^*}$. The dual measure is of the form
\begin{equation}
	\phi_{p^*, q, \mathcal D^*}^\xi(\omega^*) \propto (p^*)^{o(\omega^*)} (1-p^*)^{c(\omega^*)} q^{k((\omega^*)^\xi)}
\end{equation} 
where the dual parameter $p^*$ satisfies
\begin{align*}
	\frac{p^*p}{(1-p^*)(1-p)} = q.
\end{align*}

A path (on either the primal or dual lattice) is a sequence $(v_0, e_1, v_1, \dots, v_{N-1}, e_N, v_N)$ such that for all $k = 1,...,N$,$\|v_{k-1}-v_k\|_1=1$ and $e_k =\{v_{k-1},v_k\}$. A circuit is a path with $v_0=v_N$. If $e_k\in \mathcal E$ for all $k=1,\dots, N$ and $\omega(e_k)=1$, we say $\gamma= (e_k)_{k=1,\dots, N}$ is open; if $e_k\in \mathcal E^*$ for all $k=1,\dots, N$ and $\omega(e_k)=0$, we say $\gamma$ is a dual-closed path.

We write $B(n)$ for the domain induced by the edges in $[-n,n]^2$ and $B(x, n)$ its translation by $x\in \Z^2$. For $n_1\le n_2$, we denote annuli centered at some vertex $x$ by 
\begin{equation*}
    \mathrm{Ann}(x; n_1,n_2) = B(x, n_2) \setminus B(x, n_1).
\end{equation*}
If the annulus is centered at the origin, we drop the $x$ from the notation and instead write $\mathrm{Ann}(n_1,n_2)$.

A path of consecutive open or dual-closed edges is called an arm. A color sequence $\sigma$ of length $k$ is a sequence $(\sigma_1,\dots,\sigma_k) \in \{O,C\}^k$. Each $\sigma_i$ indicates a ``color'', with $O$ representing open and $C$ representing dual-closed. For $n_1 \le n_2$ and a vertex $x$, we define a \emph{$k$-arm event} with color sequence $\sigma$ to be the event that there are $k$ disjoint paths whose colors are specified by $\sigma$ in the annulus $\mathrm{Ann}(x; n_1,n_2)$ connecting $\partial B(x, n_1)$ to $\partial B(x, n_2)$. Formally, 
\begin{align*}
	A_{k,\sigma}(x; n_1,n_2) := \left\{ \partial B(x, n_1) \xleftrightarrow[\sigma]{\mathrm{Ann}(x; n_1,n_2)} \partial B(x,n_2)\right\}.
\end{align*}
We write $A \xleftrightarrow[\sigma_1]{\mathcal D} B$ to denote that vertex sets $A$ and $B$ are connected through a path of color $\sigma_1$ in the domain $\mathcal D$. For $A_{k,\sigma}(x; n_1,n_2)$ to occur, we let $n_0(k)$ be the smallest integer such that $|\partial B(n_0(k))| \ge k$ and let $n_1\ge n_0(k)$. Color sequences that are equivalent up to cyclic order denote the same arm event.

For this paper, there are a few special arm events: Let us fix $n$ and the boundary condition $\xi$ throughout the paper unless otherwise stated. Let $0\le n_1 < n_2 \le n$.
\begin{description}
    \item[The three-arm event] We denote by $\pi_3(n_1,n_2)$ the probability for the three-arm event $A_3(n_1,n_2) = A_{3, OOC}(n_1,n_2)$ that there are two open arms and one dual-closed arm in the annulus $\mathrm{Ann}(n_1, n_2)$:
    \begin{equation*}
        \pi_3(n_1,n_2) := \phi_{B(n)}^\xi (A_3(n_1,n_2)).
    \end{equation*}

    \item[The alternating five-arm event] There exists $c,C>0$ such that
    \begin{equation*}
        c(n_1/n_2)^2 \le \phi_{B(n)}^\xi (A_{5,OCOCO}(n_1,n_2)) \le C(n_1/n_2)^2.
    \end{equation*}
    Thus, the alternating five-arm event is said to have the universal arm exponent $2$. For proof, see \cite[Proposition 6.6]{DMT20}.
\end{description}
We remark that the dependencies on the cluster weight $q$ are implicit in the notations above.

Constants denoted by $C, c$, quantities denoted by $\alpha, \delta, \epsilon$, and boundary conditions denoted by $\eta, \iota$ are not necessarily consistent throughout the paper and their values may defer from line to line.

\subsection{Properties of the random cluster model} \label{subsec:properties}

Elaboration on the following properties can be found in \cite{Grimmett06} and \cite{DC17} or as cited.

\subsubsection*{Domain Markov property}
For any configuration $\omega' \in \{0, 1\}^E$ and any subdomain $\mathcal{F} = (W, F)$ with $F \subset E$,
\begin{equation*}
	\phi_{\mathcal{D}}^\xi (\cdot_{|F} \mid \omega_e = \omega_e',\forall e\not\in F) =\phi_{\mathcal{F}}^{\xi'} (\cdot)
\end{equation*}
where the boundary conditions $\xi'$ on $\mathcal{F}$ are defined as follows: $x$ and $y$ on $\partial\mathcal{F}$ are wired if they are connected in $\omega^\xi_{|E \setminus F}$.

\subsubsection*{Quad-crossing RSW}

\cite[Theorem 1.2]{DMT20}  Fix $1 \le q < 4$ and $p = p_c(q)$. For every $M > 0$, there exists $\eta = \eta(M) \in (0,1)$ such that for any discrete quad $(\mathcal D,a,b,c,d)$ and any boundary conditions $\xi$, if the extremal distance $\ell_{\mathcal D}[(ab),(cd)] \in [M^{-1}, M]$, then
\begin{equation} \label{eq:quad-rsw}
	\eta \le \phi^{\xi}_{\mathcal D}[(ab) \overset{\mathcal D}{\leftrightarrow}(cd)] \le 1-\eta.
\end{equation}

\subsubsection*{FKG inequality}
Fix $q \geq 1$ and a domain $\mathcal{D} = (V, E)$ of $\Z^2$. An event $A$ is called increasing if for any $\omega \le \omega'$ (for the partial order on $\{0, 1\}^E$), $\omega \in A$ implies that $\omega' \in A$. For every increasing events $A$ and $B$,
\begin{equation*}
    \phi_{\mathcal{D}}^\xi (A\cap B) \geq \phi_{\mathcal{D}}^\xi(A) \phi_{\mathcal{D}}^\xi(B).
\end{equation*}

We remark that there is no known proof for the equivalent of the generalized FKG inequality for the random cluster model.

\subsubsection*{Quasi-multiplicativity}
\cite[Proposition 6.3]{DMT20} Fix $1 \le q < 4$ and $\sigma$. There exist $c = c(\sigma,q)>0$ and $C =C(\sigma,q)>0$ such that for any boundary condition $\xi$ and every $n_0(k)\le n_1\le n_3 \le n_2$,
\begin{equation*}
    c \phi_{\mathcal D}^\xi (A_\sigma (n_1, n_2)) \le \phi_{\mathcal D}^\xi (A_\sigma (n_1, n_3)) \phi_{\mathcal D}^\xi (A_\sigma (n_3, n_2)) \le C \phi_{\mathcal D}^\xi (A_\sigma (n_1, n_2)).
\end{equation*}

\subsubsection*{Lack of Reimer's inequality}

Despite being a classical tool for Bernoulli percolation, the van den Berg-Kesten/Reimer's inequality is not known in the general form for the random cluster model, nor do we expect it to be true. A weak form of Reimer's inequality for the random cluster model is shown in \cite{BG13}. This is an issue which will be discussed in Section \ref{sec:no-reimer}.

\subsection{Lower bound for the random cluster model} \label{subsec:lower-bound}
The Aizenman-Burchard lower bound \eqref{eq:AB-lower-bd} applies when the following criterion on the probability of simultaneous traversals of separated rectangles is satisfied: A collection of rectangles ${R_j}$ is called \emph{well-separated} when the distance between any two rectangles is at least as large as the diameter of the larger. The following criterion is formulated for the random cluster measure.
\begin{Hyp}[\cite{AB99}]
	Fix $\delta>0$. There exist $\sigma>0$ and some $\rho <1$ with which for every collection of $k$ well-separated rectangles, $A_1, \dots, A_k$, of aspect ratio $\sigma$ and lengths $\ell_1, \dots, \ell_k \ge \delta n$,
	\begin{align*}
		\phi_{B(n)}^\xi 
		\begin{pmatrix}
			A_1, \dots, A_k \text{ are traversed (in} \\
			\text{the long direction) by segments} \\
			\text{of an open crossing}
		\end{pmatrix}
		\le C \rho^k.
	\end{align*}
\end{Hyp}
This hypothesis is satisfied as a consequence of the weak (polynomial) mixing property \cite{DC13}: There exists $\alpha>0$ such that for any $2\ell<m<n$ and any event $A$ depending only on edges in $B(\ell)$ and event $B$ depending only on the edges in $\mathrm{Ann}(m,n)$,
\begin{align*}
	|\phi_{B(n)}^\xi (A\cap B)  - \phi_{B(n)}^\xi (A) \phi_{B(n)}^\xi (B)| \le \left(\frac{\ell}{m}\right)^\alpha \phi_{B(n)}^\xi (A) \phi_{B(n)}^\xi (B),
\end{align*}
uniform in the boundary condition $\xi$.

\subsection{Acknowledgment}
We thank Reza Gheissari for a conversation that inspired this project. And we thank Philippe Sosoe for the many helpful discussions and detailed comments on a draft.

\section{Gluing Construction Without Generalized FKG Inequality} \label{sec:gluing}

This section is dedicated to carefully examining the gluing construction for the random cluster model. The notations used in this section are independent of the rest of the paper.

Fix $\delta>0$, a positive integer $k$, and $n_1 < n_2 < n_3$ sufficiently large. Let $E(n_1, n_2)$ be the event such that:
\begin{enumerate}
	\item there exist vertices $x_i \in \partial B(n_2)$ for $i=1,\dots, k$ and $\min_{i\neq j} |x_i- x_j| \ge 10\delta n_2$ such that $x_i$ is connected to $\partial B(x_i, 2\delta n_2) \cap B(n_2)^c$ by a path of color $\sigma_i$;
	\item $E(n_1, n_2)$ depends only on the status of the edges in $\mathrm{Ann}(n_1, n_2) \cup (\cup_{i=1}^k B(x_i, 2\delta n_2))$.
\end{enumerate}
Similarly, let $F(2n_2, n_3)$ be the event such that:
\begin{enumerate}
	\item there exist vertices $y_i \in \partial B(2n_2)$ for $i=1,\dots, k$ and $\min_{i\neq j} |y_i- y_j| \ge 20\delta n_2$ such that $y_i$ is connected to $\partial B(y_i, 2\delta n_2) \cap B(2n_2)$ by a path of color $\sigma_i$;
	\item $F(2n_2, n_3)$ depends only on the status of the edges in $\mathrm{Ann}(2n_2, n_3) \cup (\cup_{i=1}^k B(y_i, 2\delta n_2))$.
\end{enumerate}

\begin{Prop} \label{prop:gluing}
    Let $E(n_1,n_2)$, $F(2n_2, n_3)$ be as the above. Then there exists $c>0$ depending only on $k$, such that
    \begin{equation} \label{eq:gluing}
    	\begin{split}
    		\phi_{\mathrm{Ann}(n_1,n_3)}^\xi &\left(E(n_1,n_2)\cap F(2n_2,n_3)\cap \bigcap_{i=1}^k \{x_i\xleftrightarrow[\sigma_i]{\mathrm{Ann}(n_2, 2n_2)} y_i\}\right) \\
    		&\ge c\phi_{\mathrm{Ann}(n_1,n_3)}^\xi  (E(n_1,n_2)\cap F(2n_2,n_3)).
    	\end{split}
    \end{equation}
\end{Prop}

\begin{proof}
	Conditional on $E(n_1,n_2)\cap F(2n_2,n_3)$, we construct a set of $k$ corridors, $T_1, \dots, T_k$, each connecting $B(x_i, 2\delta n_2)\cap B(n_2)^c$ to $B(y_i, 2\delta n_2)\cap B(2n_2)$. Let $\gamma_1, \dots, \gamma_k$ be a collection of (topological) paths that satisfy the following constraints:
	\begin{itemize}
		\item $\gamma_i$ is a path in $\mathrm{Ann}(n_2, 2n_2)$ from $x_i$ to $y_i$.
		\item The distance between any two $\gamma_i, \gamma_j$ is at least $10\delta n_2$.
		\item The length of each $\gamma_i$ is at most $C n_2$ for some constant $C$.
	\end{itemize}
	Then, we let $T_i$ be the $\delta n_2$ neighborhood of $\gamma_i$ intersected with $\mathrm{Ann}(n_2, 2n_2)$. The $T_i$'s are disjoint by construction. We show \eqref{eq:gluing} by first dividing the right-hand side on both sides, converting the left-hand side into a conditional probability, and noting that
	\begin{align*}
		\phi_{\mathrm{Ann}(n_1,n_3)}^\xi &\left(\cap_{i=1}^k \{x_i\xleftrightarrow[\sigma_i]{\mathrm{Ann}(n_2, 2n_2)} y_i\}\: \bigg \vert \: E(n_1,n_2)\cap F(2n_2,n_3)\right) \\
		\ge &\phi_{\mathrm{Ann}(n_1,n_3)}^\xi \left(\cap_{i=1}^k \{x_i \xleftrightarrow[\sigma_i]{T_i} y_i\} \: \bigg \vert \: E(n_1,n_2)\cap F(2n_2,n_3)\right).
	\end{align*}
	It suffices to provide a constant lower bound for the right-hand side. We first use the tower rule for conditional expectations to isolate the occurrence of $\{x_1\xleftrightarrow[\sigma_1]{T_1} y_1\}$.
	\begin{align}
		\phi_{\mathrm{Ann}(n_1,n_3)}^\xi &\left(\cap_{i=1}^k \{x_i \xleftrightarrow[\sigma_i]{T_i} y_i\} \: \bigg \vert \: E(n_1,n_2)\cap F(2n_2,n_3) \right) \label{eq:cond-gluing}\\
		= \EE \bigg[\EE &\left[\mathbf{1} \left\{\cap_{i=1}^k\{x_i \xleftrightarrow[\sigma_i]{T_i} y_i\}\right\} \: \bigg \vert \: \omega |_{T_1^c}, E(n_1,n_2)\cap F(2n_2,n_3)\right ]  \: \bigg \vert \: E(n_1,n_2)\cap F(2n_2,n_3)\bigg].
	\end{align}
	Here $\EE$ denotes the expectation with respect to the measure $\phi_{\mathrm{Ann}(n_1, n_3)}^\xi$. Since $\cap_{i=2}^k \{ x_i \xleftarrow[\sigma_i]{T_i} y_i\}$ is $\omega |_{T_1^c}$-measurable, the right-hand side can be rewritten as 
	\begin{equation}\label{eq:tower-property} 
		\EE\bigg[ \EE \left[ \mathbf{1} \{x_1\xleftrightarrow[\sigma_1]{T_1} y_1\} \: \bigg \vert \: \omega |_{T_1^c}, E(n_1,n_2)\cap F(2n_2,n_3)\right]  \cdot \mathbf{1} \left\{\cap_{i=2}^k \{x_i\xleftrightarrow[\sigma_i]{T_i} y_i\}\right\} \: \bigg \vert \: E(n_1,n_2)\cap F(2n_2,n_3)\bigg].
	\end{equation}
	We write the inner conditional expectation in \eqref{eq:tower-property} back in conditional probability form as $\phi_{\mathrm{Ann}(n_1, n_3)}^\xi \left(x_1\xleftrightarrow[\sigma_1]{T_1} y_1 \: \bigg \vert \: \omega |_{T_1^c}, E(n_1,n_2)\cap F(2n_2,n_3)\right)$. Note that $\{x_1\xleftrightarrow[\sigma_1]{T_1} y_1\}$ occurs if the following events occur simultaneously:
	\begin{itemize}
		\item $\{x_1 \xleftrightarrow[\sigma_1]{B(2\delta n_2)(x_1) \cap B(n_2)^c} \partial B(x_1, 2\delta n_2) \cap B(n_2)^c\}$;
		\item $\{y_1 \xleftrightarrow[\sigma_1]{B(2\delta n_2)(y_1) \cap B(2n_2)} \partial B(y_1, 2\delta n_2) \cap B(2n_2)\}$;
		\item there is a $\sigma_1$-path in $T_1$ connecting the two short sides of $T_1$;
		\item there is a half $\sigma_1$-circuit enclosing $x_1$ in the half annulus $\mathrm{Ann}(x_1; \delta n_2, 2\delta n_2) \cap B(n_2)^c$, the event of which we denote by $\mathcal C_1$; and
		\item there is a half $\sigma_1$-circuit enclosing $y_1$ in the half annulus $\mathrm{Ann}(y_1; \delta n_2, 2\delta n_2) \cap B(2n_2)$, the event of which we denote by $\mathcal C_2$,
	\end{itemize}
	see Figure \ref{fig:gluing}.
	\begin{figure}
		\centering
		\begin{subfigure}[tb]{.5\textwidth}
			\centering\includegraphics[width=.7\linewidth]{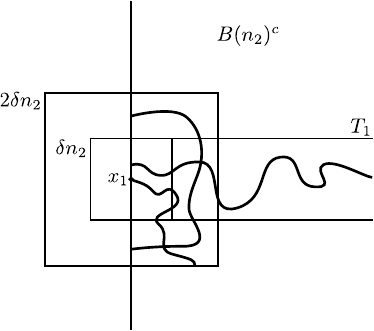}
		\end{subfigure}%
		\begin{subfigure}[tb]{.5\textwidth}
			\centering
			\includegraphics[width=.7\linewidth]{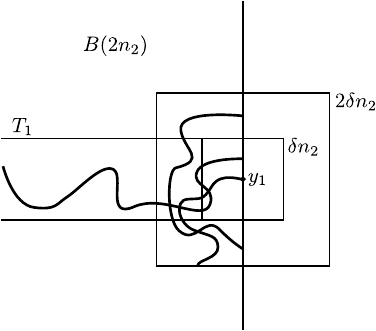}
		\end{subfigure}
		\caption{Constructions near the endpoints $x_1, x_2$.}
		\label{fig:gluing}
	\end{figure}
	Since all these events are $\sigma_1$ connection events, FKG inequality applies. To simplify notation, we denote by $\varphi(\cdot)$ the conditional measure $\phi_{\mathrm{Ann}(n_1, n_3)}^\xi(\cdot \mid \omega |_{T_1^c}, E(n_1,n_2)\cap F(2n_2,n_3))$. Then,
	\begin{align*}
		\varphi  \left(x_1\xleftrightarrow[\sigma_1]{T_1} y_1\right)
		\ge &\varphi \left( \partial B(n_2) \cap T_1 \xleftrightarrow[\sigma_1]{T_1} \partial B(2n_2) \cap T_1 \right) 
		\varphi (\mathcal C_1) 
		\varphi (\mathcal C_2) \\
		& \cdot \varphi \left(x_1 \xleftrightarrow[\sigma_1]{B(x_1, \delta n_2) \cap B(n_2)^c} \partial B(x_1, \delta n_2) \cap B(n_2)^c \right)  \\
		& \cdot \varphi \left( y_1 \xleftrightarrow[\sigma_1]{B(y_1, 2\delta n_2) \cap B(2n_2)} \partial B(y_1, 2\delta n_2) \cap B(2n_2) \right).
	\end{align*}
	The two probabilities on the last two lines are both $1$ because the occurrence of the events is guaranteed by the conditioning on $E(n_1, n_2)$ and $F(2n_2, n_3)$. The cost of half circuits is constant by RSW inequality. Since the width and length of the corridor $T_1$ is of constant proportion, again by RSW inequality, the probability of having a $\sigma_1$-path connecting the two ends of the corridor is also constant. Therefore,
	\begin{equation*}
		\phi_{\mathrm{Ann}(n_1, n_3)}^\xi \left(x_1\xleftrightarrow[\sigma_1]{T_1} y_1 \: \bigg \vert \: \omega |_{T_1^c}, E(n_1,n_2)\cap F(2n_2,n_3)\right) \ge c.
	\end{equation*}
	Plugging this back into \eqref{eq:tower-property}, we have
	\begin{align*}
		\eqref{eq:cond-gluing} 
		&\ge \EE\bigg[ c \mathbf{1} \left\{\cap_{i=2}^k \{x_i\xleftrightarrow[\sigma_i]{T_i} y_i\}\right\} \: \bigg \vert \: E(n_1,n_2)\cap F(2n_2,n_3)\bigg] \\
		&= c \EE\bigg[\mathbf{1} \left\{\cap_{i=2}^k \{x_i\xleftrightarrow[\sigma_i]{T_i} y_i\}\right\} \: \bigg \vert \: E(n_1,n_2)\cap F(2n_2,n_3)\bigg].
	\end{align*}
	Applying the same procedure sequentially to each $i = 2,\dots, k$, and we have a uniform lower bound.
\end{proof}


\section{Large Deviation Bound Conditional on Three Arms} \label{sec:cond-decor}

In this section, we prove Theorem \ref{thm:large-dev}, a large deviation bound conditional on a three-arm event. This is one of the main ingredients in the proof of Theorem \ref{thm:main}, which we outline in Section \ref{sec:outline}, and one that requires significant modification to extend to the random cluster model.

Before stating the theorem, we first introduce some notations and definitions. Fix some integer $N>0$ and for any integer $k\ge 1$, we define $\mathfrak{C}_k$ to be the event that there is a dual-closed circuit in $\mathrm{Ann}(2^{kN}, 2^{(k+1)N})$ with two defect dual-open edges. Similarly, let $\mathfrak{D}_k$ be the event that there is an open circuit in $\mathrm{Ann}(2^{kN}, 2^{(k+1)N})$ with one defect closed edge.

\begin{Definition} \label{def:circuits-k}
	For any $k\ge 1$, we define $\hat{\mathfrak{C}}_k$, the compound circuit event in $\mathrm{Ann}(2^{10kN}, 2^{10(k+1)N})$, as the simultaneous occurrence of the following events:
	\begin{enumerate}
		\item for $i = 1,4,6,9$, $\mathfrak{C}_{10k+i}$ occurs in $\mathrm{Ann}(2^{(10k+i)N}, 2^{(10k+i+1)N})$ and
		\item $\mathfrak{D}_{10k}$ occurs in $\mathrm{Ann}(2^{10kN}, 2^{(10k+1)N})$. 
	\end{enumerate}
\end{Definition}

\begin{Theorem}[\cite{DHS21}, Theorem 4.1] \label{thm:large-dev}
	There exist universal $c_1>0$ and $N_0 \ge 1$ such that for any $N\ge N_0$, any $L', L\ge 0$ satisfying $L-L'\ge 40$, and any event $E_k$ satisfying 
	\begin{enumerate}
		\item[(A)] $E_k$ depends on the status of the edges in $\mathrm{Ann}(2^{kN}, 2^{(k+1)N})$ and 
		\item[(B)] there exists a uniform constant $c_2>0$ such that $\phi_{B(n)}^\xi (E_{10k+5} \cap \hat{\mathfrak{C}}_k \mid A_3(2^{L})) \ge c_2$ for all $n\ge 0$ and $0\le k \le \frac{L}{10N}-1$.
	\end{enumerate}
	Then,
	\begin{align*}
		\phi_{B(n)}^\xi \left( \sum_{k=\lceil\frac{L'}{10N}\rceil}^{\lfloor\frac{L}{10N}\rfloor-1} \mathbf{1}\{E_{10k+5}, \hat{\mathfrak{C}}_k\}\le c_1 c_2\frac{L-L'}{N} \;\Bigg\vert\; A_3(2^{L}) \right) \le \exp(-c_1 c_2\frac{L-L'}{N}).
	\end{align*}
\end{Theorem}

For simplicity of notation, we define $I_{L',L}$ and $J_{L',L}$ to be (random) collections of indices of scales:
\begin{align*}
	I_{L',L} &:= \left\{ k = \lceil\frac{L'}{10N}\rceil, \dots, \lfloor\frac{L}{10N}\rfloor-1 : E_{10k+5} \cap \hat{\mathfrak{C}}_k \text{ occurs}\right\}. \\
	J_{L',L} &:= \left\{ k = \lceil\frac{L'}{10N}\rceil, \dots, \lfloor\frac{L}{10N}\rfloor-1 : \hat{\mathfrak{C}}_k \text{ occurs}\right\}. 
\end{align*}

The estimate in Theorem \ref{thm:large-dev} relies on a two-step strategy: first estimating $\# I_{L',L}$ using the standard Chernoff bound, then expanding the resulting expectation by conditioning on nested filtrations. Since these two steps themselves apply to general random variables, we provide a high-level summary of the proof and only recreate the parts that are sensitive to the model. 

To start, we want to condition on there existing sufficiently many decoupling circuits $\hat{\mathfrak{C}}_k$. We quantify this probability using the next Proposition which is proved later in this section.

\begin{Prop}[\cite{DHS21}, Proposition 4.2]\label{prop:enough-circ}
	There exist $c_3>0$ and $N_0\ge 1$ such that for all $N\ge N_0$ and $L,L'\ge 0$ with $L-L' \ge 40$,
	\begin{align*}
		\phi_{B(n)}^\xi \left( \# J_{L',L} \le c_3 \frac{L-L'}{N} \;\bigg\vert\; A_3(2^{L})\right)\le \exp (-c_3 (L-L')).
	\end{align*}
\end{Prop}

Combining Proposition \ref{prop:enough-circ} and the Chernoff bound, we have
\begin{align*}
	\phi_{B(n)}^\xi &\left(\# I_{L',L} \le c_4 \frac{L-L'}{N} \: \bigg \vert \: A_3(2^{L})\right) \\
	& \le \exp (-c_3 (L-L')) + \exp \left(c_4 \frac{L-L'}{N}\right) \E \left[ e^{-\# I_{L',L}} \mathbf{1}\{ \# J_{L',L} > c_3 \frac{L-L'}{N}\} \: \bigg \vert \: A_3(2^{L})\right].
\end{align*}

We decompose the expectation over all possible sets of $J_{L',L}$. 
\begin{align} \label{eq:decomp-j}
	\sum_{\mathcal J : \#\mathcal J\ge c_3 \frac{L-L'}{N}} \E \left[ e^{-\# I_{L',L}} \mid J_{L',L} = \mathcal J, A_3(2^{L})\right] \phi_{B(n)}^\xi(J_{L',L} = \mathcal J \mid A_3(2^{L})).
\end{align}
We enumerate $\mathcal J = \{k_1, \dots, k_R\}$. Then, conditional on $J_{L',L} = \mathcal J$, we have $\#I_{L',L} = \sum_{r=1}^R \mathbf{1}\{E_{10k_r+5}\}$. Define the filtration $(\mathcal F_r)$ by 
\begin{align*}
	\mathcal F_r = \sigma\{E_{10k_1+5}, \dots, E_{10k_{r-1}+5}\}\cap \{J_{L',L} = \mathcal J\} \cap A_3(2^{L}) \quad \text{for $r =1,\dots, R$.}
\end{align*}
Thus, the expectation in \eqref{eq:decomp-j} can be expanded as
\begin{align*}
	\E[e^{-\mathbf{1}\{E_{10k_1+5}\}} \cdots \E[e^{-\mathbf{1}\{E_{10k_{R-1}+5}\}} \E[e^{-\mathbf{1}\{E_{10k_R+5}\}} \mid \mathcal F_R]\mid \mathcal F_{R-1}] \cdots \mid \mathcal F_1]
\end{align*}
where for each $r = 1,\dots, R$, we have
\begin{align}\label{eq:cond-expec}
	\E[e^{-\mathbf{1}\{E_{10k_r+5}\}} \mid \mathcal F_r] = 1- (1-e^{-1}) \phi_{B(n)}^\xi (E_{10k_r+5} \mid \mathcal F_r).
\end{align}
Thus, we introduce the following lemma to give a uniform bound on the conditional probability above and decouple $E_{10k_r+5}$ from $\sigma\{E_{10k_1+5}, \dots, E_{10k_{r-1}+5}\}$ while conditional on $A_3(2^{L})$.

\begin{Lemma}\label{lemma:decoupling}
	There exists a universal constant $c_5 > 0$ such that the following holds. For any $k,L \geq 0$ and $N \geq 1$ satisfying $k\le \left\lfloor \frac{L}{10N}\right\rfloor -1$ and any events $F$ and $G$ depending on the status of edges in $B(2^{10kN})$ and $B(2^{10(k+1)N})^c$ respectively, one has
	\begin{equation} \label{eq:decoupling}
	    \phi_{B(n)}^\xi(E_{10k+5}\mid \hat{\mathfrak{C}}_k, A_3(2^{L}), F,G)\geq c_5\phi_{B(n)}^\xi(E_{10k+5}\mid \hat{\mathfrak{C}}_k, A_3(2^{L})).
	\end{equation}
\end{Lemma}

Let $k = \lceil\frac{L'}{10N}\rceil, \dots, \lfloor\frac{L}{10N}\rfloor-1$. For any $F$ depending on edges in $B(2^{10kN})$ and any $\mathcal J$ containing $k$, we have as a result of Lemma \ref{lemma:decoupling}
\begin{align*}
	\phi_{B(n)}^\xi (E_{10k+5}\mid \hat{\mathfrak{C}}_k, F, J_{L',L} = \mathcal J, A_3(2^{L})) \ge c_5\phi_{B(n)}^\xi (E_{10k+5} \mid \hat{\mathfrak{C}}_k, A_3(2^{L})) \ge c_5 \tilde{c}_0,
\end{align*}
with the second inequality owing to \eqref{eq:ej-cond-3arm} and gluing constructions (Proposition \ref{prop:gluing}) for mitigating $\hat{\mathfrak{C}}_k$. Inserting back into \eqref{eq:cond-expec}, we have
\begin{align*}
	\E[e^{-\mathbf{1}\{E_{10k_r+5}\}} \mid \mathcal F_r] \le 1- c_5\tilde c_0 (1-e^{-1})
\end{align*}
Putting everything together, we have
\begin{align*}
	\phi_{B(n)}^\xi &\left(\# I_{n',n} \le c_4 \frac{L-L'}{N}\mid A_3(2^{L})\right) \\
	&\quad \le \exp (-c_3(L-L')) + \exp\left(c_4 \frac{L-L'}{N}\right) (1-c_4\tilde c_0 (1-e^{-1}))^{c_3 \frac{L-L'}{N}}
\end{align*}
This implies the existence of some universal $c_1>0$ such that 
\begin{align*}
		\phi_{B(n)}^\xi &\left(\# I_{n',n} \le c_1 c_2 \frac{L-L'}{N}\mid A_3(2^{L})\right) \le \exp \left(-c_1 c_2 \frac{L-L'}{N}\right).
\end{align*}

\begin{proof}[Proof of Proposition \ref{prop:enough-circ}]
	We first note the set relation
	\begin{align*}
		\left\{\# J_{n',n} \le c_3 \frac{L-L'}{N}, A_3(2^{L})\right\} 
		&\subset A_3\left(2^{10N\lceil \frac{L'}{10N}\rceil}\right) \\
		& \quad \cap \left\{\bigcap_{m= 10\lceil \frac{L'}{10N}\rceil}^{10\lfloor \frac{L}{10N}\rfloor-1} A_3(2^{mN}, 2^{(m+1)N}), \#J_{L',L} \le c_3 \frac{L-L'}{N} \right\} \\
		& \quad \cap A_3\left(2^{10N\lfloor \frac{L}{10N}\rfloor}, 2^{L}\right).
	\end{align*}
	Since the three events on the right-hand side are disjoint, we apply the domain Markov property twice, first on $B(2^{10N \lceil \frac{L'}{10N}\rceil})$ and then on $B(2^{10N \lfloor \frac{L}{10N}\rfloor})$, and obtain
	\begin{align}
		\phi_{B(n)}^\xi &\left( \# J_{n',n} \le c_3 \frac{L-L'}{N}, A_3(2^{L}) \right) \nonumber\\
		&\quad \le 
		\E_{B(n)}^\xi \left[ \phi_{B(2^{10N \lceil \frac{L'}{10N}\rceil})}^{\xi'} \left( A_3\left(2^{10N \lceil \frac{L'}{10N}\rceil}\right)\right) \right] \nonumber \\
		&\quad \quad \times 
		\phi_{B(n)}^\xi \left( \bigcap_{m= 10\lceil \frac{L'}{10N}\rceil}^{10\lfloor \frac{L}{10N}\rfloor-1} A_3(2^{mN}, 2^{(m+1)N}), \#J_{L',L} \le c_3 \frac{L-L'}{N}, A_3(2^{10N\lfloor \frac{L}{10N}\rfloor}, 2^{L}) \right) \nonumber \\
		&\quad \le 
		\E_{B(n)}^\xi \left[ \phi_{B(2^{10N \lceil \frac{L'}{10N}\rceil})}^{\xi'} \left( A_3\left(2^{10N \lceil \frac{L'}{10N}\rceil}\right)\right) \right] \nonumber\\
		&\quad \quad \times 
		\E_{B(n)}^\xi \left[ \phi_{B(2^{10N \lfloor \frac{L}{10N}\rfloor})}^{\xi''} \left( \bigcap_{m= 10\lceil \frac{L'}{10N}\rceil}^{10\lfloor \frac{L}{10N}\rfloor-1} A_3(2^{mN}, 2^{(m+1)N}), \#J_{L',L} \le c_3 \frac{L-L'}{N}\right) \right] \label{eq:splitting-j-a3}\\
		&\quad \quad \times 
		\phi_{B(n)}^\xi \left(A_3(2^{10N\lfloor \frac{L}{10N}\rfloor}, 2^{L}) \right). \nonumber
	\end{align}
	Here, $\xi'$ and $\xi''$ are implicit random variables of boundary conditions on $B(2^{10N \lceil \frac{L'}{10N}\rceil})$ and $B(2^{10N \lfloor \frac{L}{10N}\rfloor})$ respectively. Since three-arm probabilities are of the same order uniform over boundary conditions, it suffices to show that \eqref{eq:splitting-j-a3} can be bounded by 
	\begin{align} \label{eq:big-o-middle-term}
		O\left(\exp (-c (L-L')) \phi_{B(2^{10N \lfloor \frac{L}{10N}\rfloor})}^{\xi} \left(A_3(2^{10N\lceil \frac{L'}{10N}\rceil}, 2^{10N \lfloor \frac{L}{10N}\rfloor})\right)\right).
	\end{align}
	Here the choice of $\xi$ is arbitrary.
	
	For each scale $m$, Let $X_m$ be the indicator function on the event $A_3(2^{10mN}, 2^{10(m+1)N})$ occurs but $\hat{\mathfrak{C}}_m$ does not. Then,
	\begin{align*}
		\left\{\bigcap_{m= 10\lceil \frac{L'}{10N}\rceil}^{10\lfloor \frac{L}{10N}\rfloor-1} A_3(2^{mN}, 2^{(m+1)N}), \#J_{L',L} \le c_3 \frac{L-L'}{N} \right\} 
		\subset 
		\left\{ \sum_{m=\lceil \frac{L'}{10N}\rceil}^{\lfloor \frac{L}{10N}\rfloor -1} X_m \ge \lfloor \frac{L}{10N}\rfloor - \lceil \frac{L'}{10N}\rceil - c_3\frac{L-L'}{N}\right\}.
	\end{align*}
	Although the $X_m$'s are not independent, when applying the domain Markov property, the dependence between events are only reflected in the boundary condition. We will establish an upper bound uniform over boundary conditions for each $X_m$, thus allowing us access to a set of independent $Y_m$'s that stochastically dominate the $X_m$'s. 
	
	\begin{Lemma}
		For each $m$, let $Y_m$ be an independent Bernoulli random variable with parameter $p_m = 5c2^{-\alpha N} \phi_{B(2^{10(m+1)N})}^\xi (A_3(2^{10mN}, 2^{10(m+1)N}))$. Then, $X_m$ is stochastically dominated by $Y_m$.
	\end{Lemma}
	
	\begin{proof}
		We give a uniform upper bound to $\mathbf P(X_m=1) = \phi_{B(2^{10(m+1)N})}^{\xi} \left(A_3(2^{10mN}, 2^{10(m+1)N}), \hat{\mathfrak{C}}_m^c\right)$. By a union bound, 
		\begin{align} \label{eq:splitting-c-hat}
			\begin{split}
				\phi_{B(2^{10(m+1)N})}^{\xi} & \left(A_3(2^{10mN}, 2^{10(m+1)N}), \hat{\mathfrak{C}}_m^c\right) \\
				& \quad = \phi_{B(2^{10(m+1)N})}^{\xi} \left(A_3(2^{10mN}, 2^{10(m+1)N}), \mathfrak{D}_{10m}^c \cup \bigcup_{i=1,4,6,9} \mathfrak{C}_{10m+i}^c \right) \\
				& \quad \le \phi_{B(2^{10(m+1)N})}^{\xi} \left( A_3(2^{10mN}, 2^{10(m+1)N}), \mathfrak{D}_{10m}^c\right) \\
				& \quad \quad + \sum_{i=1,4,6,9} \phi_{B(2^{10(m+1)N})}^{\xi} \left(A_3(2^{10mN}, 2^{10(m+1)N}), \mathfrak{C}_{10m+i}^c\right).
			\end{split}
		\end{align}
		For each $i$ and similarly for the probability with $\mathfrak{D}_{10m}^c$, we use the domain Markov property twice,
		\begin{align} \label{eq:split-domain}
			\begin{split}
				\phi_{B(2^{10(m+1)N})}^{\xi} &\left( A_3(2^{10mN}, 2^{10(m+1)N}), \mathfrak{C}_{10m+i}^c\right) \\
				& \quad \le \E_{B(2^{10(m+1)N})}^{\xi} \left[  \phi_{B(2^{(10m+i)N})}^{\eta_1}\left(A_3(2^{10mN}, 2^{(10m+i) N})\right) \right]
				\\ 
				&\quad \quad \times \E_{B(2^{10(m+1)N})}^{\xi} \left[ \phi_{B(2^{(10m+i+1)N})}^{\eta_2}\left(A_3(2^{(10m+i)N}, 2^{(10m+i+1)N}), \mathfrak{C}_{10m+i}^c\right) \right] \\
				&\quad \quad \times \phi_{B(2^{10(m+1)N})}^{\xi} \left(A_3(2^{(10m+i+1)N}, 2^{10(m+1)N})\right),
			\end{split}
		\end{align}
		where $\eta_1, \eta_2$ are random variables of boundary conditions.
		
		\begin{Claim} \label{claim:reimer}
			There exists $\alpha \in (0,1)$ and $c_6>0$ such that for any boundary condition $\eta$,  we have
			\begin{align*}
				\phi_{B(2^{(10m+i+1)N})}^\eta &\left( A_3(2^{(10m+i)N}, 2^{(10m+i+1)N}), \mathfrak{C}_{10m+i}^c \right) \\
				&\qquad \le c_6 2^{-\alpha N} \phi_{B(2^{(10m+i+1)N})}^\eta \left( A_3(2^{(10m+i)N}, 2^{(10m+i+1)N})\right), \\
				\phi_{B(2^{(10m+1)N})}^\eta &\left( A_3(2^{10mN}, 2^{(10m+1)N}), \mathfrak{D}_{10m}^c \right) \\
				&\qquad \le c_6 2^{-\alpha N} \phi_{B(2^{(10m+1)N})}^\eta \left( A_3(2^{10mN}, 2^{(10m+1)N})\right).
			\end{align*}
		\end{Claim}
		Plugging this back into \eqref{eq:split-domain} and using gluing constructions (see Section \ref{sec:gluing}), we have
		\begin{equation} \label{eq:c-comp}
			\phi_{B(2^{10(m+1)N})}^{\xi} \left( A_3(2^{10mN}, 2^{10(m+1)N}), \mathfrak{C}_{10m+i}^c\right) \le c_6 2^{-\alpha N} \phi_{B(2^{10(m+1)N})}^{\xi} \left( A_3(2^{10mN}, 2^{10(m+1)N})\right).
		\end{equation}
		Similarly, 
		\begin{equation} \label{eq:d-comp}
			\phi_{B(2^{10(m+1)N})}^{\xi} \left( A_3(2^{10mN}, 2^{10(m+1)N}), \mathfrak{D}_{10m}^c\right) \le c_6 2^{-\alpha N} \phi_{B(2^{10(m+1)N})}^{\xi} \left( A_3(2^{10mN}, 2^{10(m+1)N})\right).
		\end{equation}
		Plugging \eqref{eq:c-comp} and \eqref{eq:d-comp} into \eqref{eq:splitting-c-hat} and using quasi-multiplicativity, we have
		\begin{align*}
			\phi_{B(2^{10(m+1)N})}^{\xi} \left( A_3(2^{10mN}, 2^{10(m+1)N}), \hat{\mathfrak{C}}_m^c\right) \le c_7 2^{-\alpha N} \phi_{B(2^{10(m+1)N})}^{\xi} \left(A_3(2^{10mN}, 2^{10(m+1)N})\right).
		\end{align*}
		Choosing $p_m$ to be the right-hand side gives us that $Y_m$ stochastically dominates $X_m$.
	\end{proof}
	
	Note that there exists $\beta$ such that $p_m \in (2^{-\beta N}, 1)$ for all $m$. We use an elementary lemma from \cite{DHS21}  on the concentration of independent Bernoulli random variables.
	\begin{Lemma}[\cite{DHS21}, Lemma 4.3] \label{lemma:concentration-ber}
		Given $\eps_0 \in (0,1)$ and $M\ge 1$, if $Y_1, \dots, Y_M$ are any independent Bernoulli random variables with parameters $p_1, \dots, p_M$, respectively, satisfying $p_i \in [\eps_0, 1]$ for all $i$, then for all $r\in (0,1)$,
		\begin{align*}
			\mathbf P\left( \sum_{m=1}^M Y_m \ge rM \right) \le (1/\eps_0)^{M(1-r)} 2^M \prod_{m=1}^M p_m.
		\end{align*}
	\end{Lemma}
	
	Applying Lemma \ref{lemma:concentration-ber} by taking $M = \lfloor \frac{L}{10N}\rfloor - \lceil \frac{L'}{10N}\rceil$ and $r = 1-20c_3$, we have
	\begin{align}
		\mathbf{P} &\left( \sum_{m=\lceil \frac{L'}{10N}\rceil}^{\lfloor \frac{L}{10N}\rfloor -1} X_m \ge \lfloor \frac{L}{10N}\rfloor - \lceil \frac{L'}{10N}\rceil - c_3\frac{L-L'}{N} \right) \nonumber\\
		&\quad \quad \le \mathbf{P} \left( \sum_{m=\lceil \frac{L'}{10N}\rceil}^{\lfloor \frac{L}{10N}\rfloor -1} Y_m \ge \lfloor \frac{L}{10N}\rfloor - \lceil \frac{L'}{10N}\rceil - c_3\frac{L-L'}{N} \right) \nonumber\\
		&\quad \quad \le \left(2^{20c_3\beta N+1}\right)^{\lfloor \frac{L}{10N}\rfloor - \lceil \frac{L'}{10N}\rceil} \prod_{m=\lceil \frac{L'}{10N}\rceil}^{\lfloor \frac{L}{10N}\rfloor -1} p_m. \label{eq:prod-pm}
	\end{align}

	Plugging in $p_m = c_7 2^{-\alpha N} \phi_{B(2^{10(m+1)N})}^{\xi} \left(A_3(2^{10mN}, 2^{10(m+1)N})\right)$, we have
	\begin{align*}
		\eqref{eq:prod-pm} \le \left(2^{20c_3\beta N+1}\right)^{\lfloor \frac{L}{10N}\rfloor - \lceil \frac{L'}{10N}\rceil} C 2^{-\alpha N} \phi_{B(2^{10N\lfloor \frac{L}{10N}\rfloor})}^{\xi} \left(A_3(2^{10N\lceil \frac{L'}{10N}\rceil}, 2^{10N\lfloor \frac{L}{10N}\rfloor})\right),
	\end{align*}
	which demonstrates \eqref{eq:big-o-middle-term}.
\end{proof}

\begin{proof}[Proof of Claim \ref{claim:reimer}]
	By duality or Menger's theorem, the event $A_3(2^{kN}, 2^{(k+1)N})\cap \mathfrak{C}_{k}^c$ (resp. the event $A_3(2^{kN}, 2^{(k+1)N})\cap \mathfrak{D}_{k}^c$) is equivalent to the disjoint occurrence of $A_3(2^{kN}, 2^{(k+1)N})$ and $A_{1,O}(2^{kN}, 2^{(k+1)N})$ (resp. $A_{1,C^*}(2^{kN}, 2^{(k+1)N})$). We denote the disjoint occurrence of two events $A$ and $B$ by $A\circ B$. We have
	\begin{align*}
		\phi_{B(2^{(k+1)N})}^\eta &\left( A_3(2^{kN}, 2^{(k+1)N}) , \mathfrak{C}_k^c \right) \\
		&\qquad = \phi_{B(2^{(k+1)N})}^\eta \left(A_3(2^{kN}, 2^{(k+1)N}) \circ A_{1,O}(2^{kN}, 2^{(k+1)N})\right) 
	\end{align*}
	
	We use a now standard argument that conditional on the two open and one dual-closed arm, the probability that an additional arm exists decays at $(2^N)^{-\alpha}$ for some $\alpha \in (0,1)$. This argument requires \emph{localizing} the endpoints of the arms. 
	
	Let $I = \{I_i\}_{i=1,2,3}$ and $J = \{J_i\}_{i=1,2,3}$ be two sequences of disjoint intervals on $\partial B(2^{kN})$ and $\partial B(2^{(k+1)N})$ in counterclockwise order such that for all $i$, $|I_i| \ge \delta 2^{kN}$ for some $\delta>0$. Similarly, for all $i$, $|J_i| \ge \delta 2^{(k+1)N}$. Let $A_3^{I, J}(2^{kN}, 2^{(k+1)N})$ be the event that $A_3(2^{kN}, 2^{(k+1)N})$ occurs and there is an open (open, dual-closed, resp.) arm with endpoints in $I_1$ ($I_2, I_3$, resp.) on $\partial B(2^{kN})$ and $J_1$ ($J_2, J_3$, resp.) on $\partial B(2^{(k+1)N})$. Localization of arm events \cite[Proposition 6.5]{DMT20} implies that for any $I, J$ as above,
	\begin{align*}
		\phi_{B(2^{(k+1)N})}^\eta \left(A_3(2^{kN}, 2^{(k+1)N}) \right) \asymp \phi_{B(2^{(k+1)N})}^\eta \left(A_3^{I,J}(2^{kN}, 2^{(k+1)N}) \right)
	\end{align*}
	and
	\begin{align*}
		\phi_{B(2^{(k+1)N})}^\eta &\left(A_3(2^{kN}, 2^{(k+1)N}) \circ A_{1,O}(2^{kN}, 2^{(k+1)N}) \right) \\
		& \qquad \asymp \phi_{B(2^{(k+1)N})}^\eta \left(A_3^{I,J}(2^{kN}, 2^{(k+1)N}) \circ A_{1,O}(2^{kN}, 2^{(k+1)N}) \right).
	\end{align*}
	
	Let $\gamma_1$ be the counterclockwise-most open arm from $I_1$ to $J_1$ that is disjoint from at least one open arm from $I_2$ to $J_2$ and let $\gamma_3$ be the clockwise-most dual-closed arm from $I_3$ to $J_3$. Let $\mathcal U$ denote the random region bounded by $\gamma_1$, $\gamma_3$, $\partial B(2^{kN})$, and $\partial B(2^{(k+1)N})$ that also contains an open arm from $I_2$ to $J_2$. We say that $U$ is ``admissible'' if $U$ is a possible value of the random region $\mathcal U$. Conditional on the location of $\mathcal U$,
	\begin{align*}
	\phi_{B(2^{(k+1)N})}^\eta &\left(A_3^{I,J}(2^{kN}, 2^{(k+1)N}) \circ A_{1,O}(2^{kN}, 2^{(k+1)N})\right) \\
		&= \sum_{\text{admissible } U} \E_{B(2^{(k+1)N})}^\eta \left[\phi_{B(2^{(k+1)N}) \setminus U}^\iota \left(A_{1,O}(2^{kN}, 2^{(k+1)N}) \right)\right] \phi_{B(2^{(k+1)N})}^\eta(\mathcal U = U)
	\end{align*}
	where the boundary condition $\iota$ is a random variable depending on $U$ and $\eta$. The one-arm event in $B(2^{(k+1)N}) \setminus U$ can be estimated using quad-crossing RSW estimates \eqref{eq:quad-rsw} and an analogue of quasi-multiplicativity $N$ times as follows:
	\begin{align*}
		\phi_{B(2^{(k+1)N}) \setminus U}^\iota (A_{1,O}(2^{kN}, 2^{(k+1)N})) 
		\le c \prod_{\ell= kN}^{(k+1)N} \phi_{B(2^{(k+1)N}) \setminus U}^\iota (A_1(2^\ell, 2^{\ell+1})) 
		\le c (2^{-\alpha})^{N}.
	\end{align*}
	Thus, 
	\begin{align*}
		\phi_{B(2^{(k+1)N})}^\eta \left( A_3(2^{kN}, 2^{(k+1)N}) , \mathfrak{C}_k^c \right) \le c 2^{-\alpha N} \phi_{B(2^{(k+1)N})}^\eta \left(A_3(2^{kN}, 2^{(k+1)N})\right).
	\end{align*}
	By the same reasoning, 
	\begin{align*}
		\phi_{B(2^{(k+1)N})}^\eta \left( A_3(2^{kN}, 2^{(k+1)N}) , \mathfrak{D}_k^c \right) \le c' 2^{-\alpha N} \phi_{B(2^{(k+1)N})}^\eta \left(A_3(2^{kN}, 2^{(k+1)N})\right).
	\end{align*}
\end{proof}

Now we prove Lemma \ref{lemma:decoupling}.
\begin{proof}[Proof of Lemma \ref{lemma:decoupling}]
	Recall that $F$ and $G$ depend on the status of edges in $B(2^{10kN})$ and $B(2^{10(k+1)N})^c$ respectively. We focus on demonstrating how we remove the conditioning on $F$ as the other side is similar.
	\begin{align} \label{eq:decouple-f}
		\phi_{B(n)}^\xi (E_{10k+5} \mid \hat{\mathfrak{C}}_k, A_3(2^{L}), F, G) \ge c_8 \phi_{B(n)}^\xi (E_{10k+5} \mid \hat{\mathfrak{C}}_k, A_3(2^{L}), G).
	\end{align}
	
	Recall from Definition \ref{def:circuits-k} that $\hat{\mathfrak{C}}_k$ is the simultaneous occurrence of a stack of circuits. We use these circuits to separate $F$ and $E_{10k+5}$ while conditioning on $A_3(2^{L})$. 
	
	A dual-closed circuit $\mathcal C$ with two defect edges is naturally divided into two arcs between these defects consisting of open connections. Fix some deterministic ordering of arcs and label the two arcs $\mathrm{Arc}_1(\mathcal C)$ and $\mathrm{Arc}_2(\mathcal C)$ in this ordering. Let $X_- (\mathcal C, i)$ be the event such that:
	\begin{enumerate}
		\item $\mathfrak{D}_{10k}$ occurs;
		\item $\mathcal C$ is the \emph{innermost} dual-closed circuit with two defect edges in $\mathrm{Ann}(2^{(10k+1)N}, 2^{(10k+2)N})$;
		\item the origin is connected to the two defects in $\mathcal C$ through two disjoint open paths;
		\item $(\frac12, -\frac12)$ is connected to $\mathrm{Arc}_i(\mathcal C)$ through a dual-closed path.
	\end{enumerate} 
	Note that the occurrence of $\mathfrak{D}_{10k}$, together with the two open paths through the origin, guarantees that item (4) only occurs for one of the two arcs. Hence $X_-(\mathcal C, i)$ occurs for exactly one choice of $\mathcal C$ and $i$.
	Similarly, let $X_+ (\mathcal D, j)$ be the event such that:
	\begin{enumerate}
		\item $\mathfrak{C}_{10k+i}$ occurs for $i=6,9$;
		\item $\mathcal D$ is the \emph{innermost} dual-closed circuit with two defect edges in $\mathrm{Ann}(2^{(10k+4)N}, 2^{(10k+5)N})$;
		\item the two defects are connected to $\partial B(2^{L})$ through disjoint open paths;
		\item $\mathrm{Arc}_j(\mathcal D)$ is connected to $\partial B(2^{L})$ through a dual-closed path.
	\end{enumerate}
	We also need a three-arm event between $\mathcal C$ and $\mathcal D$. The dual-closed arm connects an arc of $\mathcal C$ and an arc of $\mathcal D$ and the indices of the arcs are important. Thus, let $X(\mathcal C, \mathcal D, i, j)$ be the event such that:
	\begin{enumerate}
		\item there is a dual-closed path connecting $\mathrm{Arc}_i(\mathcal C)$ and $\mathrm{Arc}_j(\mathcal D)$ in the region between $\mathcal C$ and $\mathcal D$;
		\item there is a pair of disjoint open paths in the region between $\mathcal C$ and $\mathcal D$ connecting a defect of $\mathcal C$ and a defect of $\mathcal D$.
	\end{enumerate}
	Note that for each pair of $i$ and $j$, topologically there is only one possible way to connect the defects of $\mathcal C$ and $\mathcal D$ with open paths.
	
	Then, the occurrence of $\hat{\mathfrak{C}}_k$ allows for a decomposition in admissible $\mathcal C$, $\mathcal D$, and $i,j = 1,2$ for the following event.
	\begin{align*}
		\phi_{B(n)}^\xi &(E_{10k+5}, \hat{\mathfrak{C}}_k, A_3(2^{L}), F, G) \\
		&\quad = \sum_{\mathcal C, \mathcal D, i, j}\phi_{B(n)}^\xi (F, X_-(\mathcal C, i), X(\mathcal C, \mathcal D, i, j), X_+(\mathcal D, j), E_{10k+5}, G).
	\end{align*}
	
	Since $\mathcal C$ is the innermost circuit, its position can be determined via exploration in the interior of $\mathcal C$. Using the domain Markov property first on $\mathrm{Ext}(\mathcal C) := B(n) \setminus \mathrm{Int} (\mathcal C)$, the boundary condition induced by the configuration in $\mathrm{Int}(\mathcal C)$ identifies only the two defect edges in $\mathcal C$, which we denote by $0^*(\mathcal C)$. Thus, we have 
	\begin{align*}
		\phi_{B(n)}^\xi &(F, X_-(\mathcal C, i), X(\mathcal C, \mathcal D, i, j), X_+(\mathcal D, j), E_{10k+5}, G) \\
		&\quad = \phi_{B(n)}^\xi (F, X_-(\mathcal C, i)) \phi_{\mathrm{Ext}(\mathcal C)}^{0^*(\mathcal C)} (X(\mathcal C, \mathcal D, i, j), X_+(\mathcal D, j), E_{10k+5}, G).
	\end{align*}
	Using the domain Markov property again on $\mathrm{Ext}(\mathcal D)$, the configuration in $\mathrm{Int}(\mathcal D) \setminus \mathrm{Ext}(\mathcal C)$ induces a free boundary condition on $\mathcal D$. Thus, we have
	\begin{align}
		\phi_{B(n)}^\xi &(E_{10k+5}, \hat{\mathfrak{C}}_k, A_3(2^{L}), F, G) \nonumber\\
		&\quad = \sum_{\mathcal C, \mathcal D, i, j} \phi_{B(n)}^\xi (F, X_-(\mathcal C, i)) \phi_{\mathrm{Ext}(\mathcal C)}^{0^*(\mathcal C)} (X(\mathcal C, \mathcal D, i, j)) \phi_{\mathrm{Ext}(\mathcal D)}^{0} (X_+(\mathcal D, j), E_{10k+5}, G). \label{eq:dom-mar-e-fg}
	\end{align}
	where $0$ is the free boundary condition.
	A similar decomposition gives, 
	\begin{align}
		\phi_{B(n)}^\xi &(\hat{\mathfrak{C}}_k, A_3(2^{L}), G) \nonumber\\
		&\quad = \sum_{\mathcal C', \mathcal D', i', j'} \phi_{B(n)}^\xi (X_-(\mathcal C', i')) \phi_{\mathrm{Ext}(\mathcal C')}^{0^*(\mathcal C')}(X(\mathcal C', \mathcal D', i', j')) \phi_{\mathrm{Ext}(\mathcal D')}^{0} (X_+(\mathcal D', j'), G). \label{eq:dom-mar-g}
	\end{align}
	Multiplying \eqref{eq:dom-mar-e-fg} and \eqref{eq:dom-mar-g} gives
	\begin{align} \label{eq:dom-mar-multiplied}
		\begin{split}
			\phi_{B(n)}^\xi &(E_{10k+5}, \hat{\mathfrak{C}}_k, A_3(2^{L}), F, G) \phi_{B(n)}^\xi (\hat{\mathfrak{C}}_k, A_3(2^{L}), G) \\
			&=\sum_{\substack{\mathcal C, \mathcal D, i, j \\ \mathcal C', \mathcal D', i', j'}} \Big[ \phi_{B(n)}^\xi (F, X_-(\mathcal C, i)) \phi_{\mathrm{Ext}(\mathcal C)}^{0^*(\mathcal C)} (X(\mathcal C, \mathcal D, i, j)) \phi_{\mathrm{Ext}(\mathcal D)}^{0} (X_+(\mathcal D, j), E_{10k+5}, G) \\
			&\qquad \times \phi_{B(n)}^\xi (X_-(\mathcal C', i')) \phi_{\mathrm{Ext}(\mathcal C')}^{0^*(\mathcal C')}(X(\mathcal C', \mathcal D', i', j')) \phi_{\mathrm{Ext}(\mathcal D')}^{0} (X_+(\mathcal D', j'), G)\Big].
		\end{split}
	\end{align}
	We use the following estimate which is the random cluster analogue of \cite[Lemma 6.1]{DS11}. It is essentially a corollary of the so-called strong separation lemmas, the random cluster version of which we prove in Section \ref{sec:ext-sep}. To get from the strong separation lemmas to the following estimate, a proof sketch of no essential difference can be found in \cite{DS11}. 
	
	There exists a uniform constant $c_9$ such that the following holds for all choices of circuits $\mathcal C, \mathcal C', \mathcal D, \mathcal D'$ and arc indices $i, i', j, j'$:
	\begin{align} \label{eq:dam-sap-comp}
		\frac{\phi_{\mathrm{Ext}(\mathcal C)}^{0^*(\mathcal C)}(X(\mathcal C, \mathcal D', i, j')) \phi_{\mathrm{Ext}(\mathcal C')}^{0^*(\mathcal C')}(X(\mathcal C', \mathcal D, i', j))}{\phi_{\mathrm{Ext}(\mathcal C)}^{0^*(\mathcal C)}(X(\mathcal C, \mathcal D, i, j)) \phi_{\mathrm{Ext}(\mathcal C')}^{0^*(\mathcal C')}(X(\mathcal C', \mathcal D', i', j'))} < c_9.
	\end{align}
	Applying \eqref{eq:dam-sap-comp} to the summand of \eqref{eq:dom-mar-multiplied}, we have
	\begin{align*}
		\phi_{B(n)}^\xi &(F, X_-(\mathcal C, i)) \phi_{\mathrm{Ext}(\mathcal C)}^{0^*(\mathcal C)} (X(\mathcal C, \mathcal D, i, j)) \phi_{\mathrm{Ext}(\mathcal D)}^{0} (X_+(\mathcal D, j), E_{10k+5}, G)\\
		&\qquad \times \phi_{B(n)}^\xi (X_-(\mathcal C', i')) \phi_{\mathrm{Ext}(\mathcal C')}^{0^*(\mathcal C')}(X(\mathcal C', \mathcal D', i', j')) \phi_{\mathrm{Ext}(\mathcal D')}^{0} (X_+(\mathcal D', j'), G) \\
		> &c_9^{-1}  \phi_{B(n)}^\xi (F, X_-(\mathcal C, i)) \phi_{\mathrm{Ext}(\mathcal C)}^{0^*(\mathcal C)}(X(\mathcal C, \mathcal D', i, j')) \phi_{\mathrm{Ext}(\mathcal D')}^{0} (X_+(\mathcal D', j'), G) \\
		&\qquad \times \phi_{B(n)}^\xi (X_-(\mathcal C', i')) \phi_{\mathrm{Ext}(\mathcal C')}^{0^*(\mathcal C')}(X(\mathcal C', \mathcal D, i', j)) \phi_{\mathrm{Ext}(\mathcal D)}^{0} (X_+(\mathcal D, j), E_{10k+5}, G).
	\end{align*}
	Summing over $\mathcal C, \mathcal D, i, j, \mathcal C', \mathcal D', i', j'$, by the domain Markov property,
	\begin{align*}
		\eqref{eq:dom-mar-multiplied} > & c_9^{-1} \sum_{\substack{\mathcal C, \mathcal D, i, j \\ \mathcal C', \mathcal D', i', j'}} \Big[ \phi_{B(n)}^\xi (F, X_-(\mathcal C, i)) \phi_{\mathrm{Ext}(\mathcal D')}^{0} (X_+(\mathcal D', j'), G) \phi_{\mathrm{Ext}(\mathcal C)}^{0^*(\mathcal C)}(X(\mathcal C, \mathcal D', i, j')) \\
		&\qquad \times \phi_{B(n)}^\xi (X_-(\mathcal C', i')) \phi_{\mathrm{Ext}(\mathcal D)}^{0} (X_+(\mathcal D, j), E_{10k+5}, G) \phi_{\mathrm{Ext}(\mathcal C')}^{0^*(\mathcal C')}(X(\mathcal C', \mathcal D, i', j))\Big] \\
		= & c_9^{-1} \phi_{B(n)}^\xi (\hat{\mathfrak{C}}_k, A_3(2^{L}), F, G) \phi_{B(n)}^\xi (E_{10k+5}, \hat{\mathfrak{C}}_k, A_3(2^{L}), G).
	\end{align*}
	Dividing both sides by $\phi_{B(n)}^\xi (\hat{\mathfrak{C}}_k, A_3(2^{L}), G) \phi_{B(n)}^\xi (\hat{\mathfrak{C}}_k, A_3(2^{L}), F, G)$, we have \eqref{eq:decouple-f} with $c_8= c_9^{-1}$.
	
	From \eqref{eq:decouple-f}, we can remove the conditioning on $G$ using a nearly identical argument.
\end{proof}

Although the above proof is formally similar to the proof of Lemma 4.4 in \cite{DHS21}, it heavily relies on the domain Markov property, so the choice of the domain and the order of application are crucial.


\section{Arm Separation for the Random Cluster Model} \label{sec:ext-sep}

As indicated in the proof of Lemma \ref{lemma:decoupling}, \eqref{eq:dam-sap-comp} depends on the following two strong arm separation lemmas in combination with the gluing constructions explained in Section \ref{sec:gluing}. 

\begin{Lemma}[External Arm Separation] \label{lemma:ext-arm-sep}
	Fix an integer $m\ge 2$ and let $n_1 \le n_2 -3$. Consider an open circuit $\mathcal C$ in $B(2^{n_1})$ with $m$ defects $e_1, \dots, e_m$.
	Let $\mathcal A(\mathcal C, 2^{n_2})$ be the event that 
	\begin{enumerate}[$(1)$]
		\item there are $2m$ alternating disjoint open arms and dual-closed arms from $\mathcal C$ to $\partial B(2^{n_2})$ in $B(2^{n_2})\setminus \mathrm{Int}(\mathcal C)$;
		\item the $m$ dual-closed paths emenate from $e_j^*$ to $\partial B(2^{n_2})^*$, respectively.
	\end{enumerate}
	We note that the locations of the defects $e_1, \dots, e_m$ are implicit in the notation of $\mathcal C$.
	Let  $\tilde{\mathcal A}(\mathcal C, 2^{n_2})$ be the event that $\mathcal A(\mathcal C, 2^{n_2})$ occurs with $2m$ arms $\gamma_1, \dots, \gamma_{2m}$ (open, dual-closed alternatingly) whose endpoints in $\partial B(2^{n_2})$ or $\partial B(2^{n_2})^*$, $f_1, \dots, f_{2m}$, satisfy
	\begin{equation*}
		2^{-n_2} \min_{k\neq l}|f_k - f_l|  \ge \frac1{2m}.
	\end{equation*}
	Then, there is a constant $c_{10}(m)>0$ independent of $n_1, n_2, \mathcal C$, and the boundary condition $\xi$ such that 
	\begin{equation} \label{eq:ext-arm-sep}
		\phi_{B(2^{n_2})}^\xi (\mathcal A(\mathcal C,2^{n_2})) \le c_{10}(m) \phi_{B(2^{n_2})}^{\xi'} (\tilde {\mathcal A}(\mathcal C,2^{n_2}))
	\end{equation}
	for some boundary condition $\xi'$ on $B(2^{n_2})$.
\end{Lemma}

\begin{Lemma}[Internal Arm Separation] \label{lemma:int-arm-sep}
	Fix an integer $m\ge 2$ and let $n_3 +3 \le n_4$. Consider an open circuit $\mathcal D$ in $B(2^{n_4})^c$ with $m$ defects $g_1, \dots, g_m$.
	Let $\mathcal B(2^{n_3}, \mathcal D)$ be the event that 
	\begin{enumerate}[$(1)$]
		\item there are $2m$ alternating disjoint open arms and dual-closed arms from $\partial B(2^{n_3})$ to $\mathcal D$ in $\mathrm{Int}(\mathcal D)\setminus B(2^{n_3})$;
		\item the $m$ dual-closed paths emenate from $g_j^*$ to $\partial B(2^{n_3})^*$, respectively.
	\end{enumerate}
	Let  $\tilde{\mathcal B}(2^{n_3},\mathcal D)$ be the event that $\mathcal B(2^{n_3}, \mathcal D)$ occurs with $2m$ arms $\gamma_1, \dots, \gamma_{2m}$ (open, dual-closed alternatingly) whose endpoints in $\partial B(2^{n_3})$ or $\partial B(2^{n_3})^*$, $h_1, \dots, h_{2m}$, satisfy
	\begin{equation*}
		2^{-n_3} \min_{k\neq l}|h_k - h_l|  \ge \frac1{2m}.
	\end{equation*}
	Then, there is a constant $c_{11}(m)>0$ independent of $n_3, n_4, \mathcal D$, and the boundary condition $\xi$ such that 
	\begin{equation} \label{eq:ext-arm-sep}
		\phi_{\mathrm{Int}(\mathcal D)}^\xi (\mathcal B(2^{n_3}, \mathcal D)) \le c_{11}(m) 	\phi_{\mathrm{Int}(\mathcal D)}^{\xi'} (\tilde {\mathcal B}(2^{n_3}, \mathcal D))
	\end{equation}
	for some boundary condition $\xi'$ on $\mathcal D$.
\end{Lemma}

\begin{Remark}
	$\xi'$ arises due to a technical challenge in the proof. However, for the purpose of \eqref{eq:dam-sap-comp}, any boundary condition suffices as the RSW estimates we have are uniform in boundary conditions. 
\end{Remark}

The proofs of Lemma \ref{lemma:ext-arm-sep} and \ref{lemma:int-arm-sep} are similar, and we only provide the proof for the former. 

Arm separation techniques are classical techniques that date back to Kesten \cite{Kesten87, Nolin08}. They were first developed to show well-separatedness for arms crossing square annuli. In our case, the annulus consists of one square boundary and one circuitous boundary. The main obstacle for directly applying the classical arm separation arguments is that the geometry of the circuit may generate bottlenecks that prevent arms from being separated on certain scales. In the first part of the proof, we address this through a construction that ``leads'' the interfaces to the boundary of $B(2^{n_1})$. We note that this part of the proof for the random cluster model is identical to that of \cite[Lemma 6.2]{DS11} as the constructions are purely topological. However, we include here for the reader's convenience. The second step is to define a family of disjoint annuli in levels, which groups the arms based on their relative distances. In the following proof, the details for this step is provided last. The final part of the proof depends on an arm separation statement in each annuli defined in the previous step, for which we provide the details in Lemma \ref{lemma:sep-fixed-scale}.

We note that our proof is stated in full generality compared to the proof in \cite{DS11} which is stated for $m=2$, and therefore slightly deviates from it in notation.

\begin{proof}[Proof of Lemma \ref{lemma:ext-arm-sep}]
	
	Given the circuit $\mathcal C$ with defects $e_1, \dots, e_m$, we assume the occurrence of $\mathcal A(\mathcal C, 2^{n_2})$. The first step is to ``extend'' the circuit $\mathcal C$ to $\partial B(2^{n_1})$ so that the arms will not be tangled due to the geometry of $\mathcal C$.
	
	For $i=1,\dots, m$, $\alpha_i^l$ be the counterclockwise-most dual-closed path emenating from $e_i^*$ to $\partial B(2^{n_1}+1/2)$ in $B(2^{n_1}+1/2) \setminus \mathcal C$ and $\alpha_i^r$ the clockwise-most dual-closed path emenating from $e_i^*$ to $\partial B(2^{n_1}+1/2)$ in $B(2^{n_1}+1/2) \setminus \mathcal C$. We denote by $a_i^l$ the first vertex on $\partial B(2^{n_1})$ to the counterclockwise side of $\alpha_i^l$ and $a_i^r$ the first vertex on $\partial B(2^{n_1})$ to the clockwise side of $\alpha_i^r$. Let $\beta_i^l$ be the counterclockwise-most open path from the lower right end-vertex of $e_i$ to $a_i^r$ in $B(2^{n_1}) \setminus \mathcal C$ and $\beta_i^r$ be the clockwise-most open path from the top left end-vertex of $e_{i+1}$ to $a_{i+1}^l$ in $B(2^{n_1}) \setminus \mathcal C$. Here, the indices are cyclic, meaning that $i= i \bmod m$. 
	
	Note that it is necessary that $a_i^l \neq a_i^r$; it is possible that $a_i^r = a_{i+1}^l$, but by the assumption that $\mathcal A(\mathcal C, n_2)$ occurs, $a_{i+1}^l$ must be on the clockwise side of $a_i^r$ on $\partial B(2^{n_1})$.
	
	We identify the last intersection of $\alpha_i^l$ and $\alpha_i^r$ from $e_i$ to $\partial B(2^{n_1})$, which can possibly be $e_i$. Let $\alpha_i$ be the union of the piece of $\alpha_i^l$ from the last intersection to $\partial B(2^{n_1}+1/2)$ and the piece of $\alpha_i^r$ from the last intersection to $\partial B(2^{n_1}+1/2)$. Let $R_i$ be the domain bounded by $\alpha_i$ and the piece of $\partial B(2^{n_1})$ between $a_i^l$ and $a_i^r$ on $a_i^l$'s clockwise side. We now define a path $\beta_i$. If $\beta_i^l$ and $\beta_i^r$ intersect, we define $\beta_i$ analogously to $\alpha_i$. Otherwise, we define $\beta_i$ to be the union of the piece of $\beta_i^l$ from its last intersection with $\mathcal C$ to $\partial B(2^{n_1})$, the piece of $\beta_i^r$ from its last intersection with $\mathcal C$ to $\partial B(2^{n_1})$, and the piece of $\mathcal C$ that connects the aforementioned two pieces. Let $S_i$ be the domain bounded by $\beta_i$ and the piece of $\partial B(2^{n_1})$ between $a_i^r$ and $a_{i+1}^l$ on $a_i^r$'s clockwise side. Note that in the case $a_i^r = a_{i+1}^l$, $\beta_i$ and $S_i$ consist of only the vertex $a_i^r$.
	
	Let $R := (B(2^{n_2}) \setminus B(2^{n_1})) \cup (\cup_{i=1}^m R_i) \cup (\cup_{i=1}^m S_i)$. Note that once $\{\alpha_i, \beta_i\}_i$ is fixed, the conditional distribution of the cluster configuration inside $R$ is (uniquely) determined by the status of $\alpha_i$ and $\beta_i$. Let $\mathcal A(R)$ denote the event that
	\begin{enumerate}
		\item there is a dual-closed arm connecting $\alpha_i$ to $\partial B(2^{n_2})^*$ in $R$ for $i=1,\dots, m$;
		\item there is an open arm connecting $\beta_i$ to $\partial B(2^{n_2})$ in $R$ for $i=1, \dots, m$.
	\end{enumerate}
	Let  $\tilde{\mathcal A}(R)$ be the event that $\mathcal A(R)$ occurs with $2m$ arms $\gamma_1, \dots, \gamma_{2m}$ (dual-closed, open alternatingly) whose endpoints in $\partial B(2^{n_2})$ or $\partial B(2^{n_2})^*$, $f_1, \dots, f_{2m}$, satisfy $2^{-n_2} \min_{k\neq l}|f_k - f_l|  \ge 1/(2m)$. Lemma \ref{lemma:ext-arm-sep} is then equivalent to 
	\begin{equation*}
		\phi_{B(2^{n_2})}^\xi (\mathcal A(R)) \le c'_{10}(m) \phi_{B(2^{n_2})}^{\xi'}(\tilde{\mathcal A}(R))
	\end{equation*}
	for some boundary condition $\xi'$ and some constant $c'_{10}(m)>0$ that only depends on $m$. 
	
	\begin{figure} 
		\centering
		\includegraphics[width=0.55\textwidth]{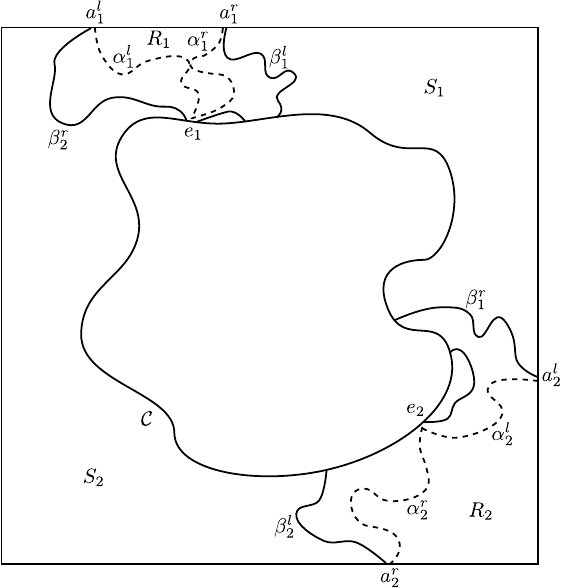}
		\caption{A representation of the construction in the first step with $m=2$. This figure is topologically equivalent to \cite[Fig. 4]{DS11}, with a relabeling.}
		\label{fig:arm-sep}
	\end{figure}
	
	Let us first relabel the vertices $a_1^l, a_1^r, \dots, a_m^l, a_m^r$ by $x_1, \dots, x_{2m}$ where $x_{2i-1} = a_i^l$ and $x_{2i} = a_i^r$ for $i=1, \dots, m$. The next step is to identify \emph{critical scales}, scales of neighborhoods of these vertices comparable to the distance beween them. We now informally introduce the notion of \emph{level-$j$ annuli} so we can finish the proof before returning to formally defining them at the end of the proof. 
	
	For $j= 1,\dots, 2m-1$, $\mathcal I_j$ is a collection of indices. $\mathcal I_j$ keeps track of groups of $j+1$ vertices among $x_1, \dots, x_{2m}$ on level $j$ and the index indicates the first vertex in a group in clockwise order. For each level $j$, the difference of two indices in $\mathcal I_j$ is at least $j+1$.
	
	Let $\mathcal L_j$ be the collection of level-$j$ annuli: $\mathcal L_j := \{\mathrm{Ann}_j (i): i\in \mathcal I_j\}$. The level-$j$ annuli $\mathrm{Ann}_j(i)$ satisfy:
	\begin{itemize}
		\item If $i \in \mathcal I_j$, that is,  $\mathrm{Ann}_j(i)$ is nonempty, then $\mathrm{Ann}_j(i)$ is centered on $\partial B(2^{n_1})$ and both its inner box and outer box enclose exactly $j+1$ vertices $x_i, x_{i+1}, \dots, x_{i+j}$. That is, $\mathrm{Ann}_j(i)$ is crossed by $j$ arms.
		\item Level-$j$ annuli are mutually disjoint and disjoint from annuli of other levels.
		\item There is exactly one level-$2m$ annulus and at most $\lfloor \frac{2m}{j+1} \rfloor$ (and possibly zero) level-$j$ annuli.
		\item All level-$j$ annuli are contained in $B(2^{n_1+1})$, for $j=1, \dots, 2m-1$. The level-$2m$ annulus is $\mathrm{Ann}_{2m}(1) = \mathrm{Ann}(2^{n_1+1}, 2^{n_2})$.
	\end{itemize}
	
	$\mathcal A(R)$ implies the simultaneous occurrence of crossings in each of the annuli defined above intersected with the domain $R$, which can cause the annuli to have irregular boundaries. However, since all annuli (excluding the level-$2m$ annulus) are centered on $\partial B(2^{n_1})$ and the boundary of $R$ (the $\alpha_i$ and $\beta_i$) are in the interior of $B(2^{n_1})$, each annulus $\mathrm{Ann}_j(i)$ intersected with $R$ necessarily contains one of the top-, bottom-, left-, or right-half of $\mathrm{Ann}_j(i)$. We call the half annulus $\mathrm{Ann}_j^\mathrm{h}(i)$. If there are two choices, choose the top or bottom over the left or right.
	
	For $j=1, \dots, 2m-1$, let $E_j(i)$ be the event that there exist $j$ disjoint crossings in $\mathrm{Ann}_j^\mathrm{h}(i)$ such that the color of each crossing is determined by the vertices the annulus encloses. In particular, let $E_{2m}$ be the event that $\mathrm{Ann}_{2m}$ is crossed by $2m$ disjoint alternating open and dual-closed crossings. Then $\mathcal A (R)$ implies the occurence of $\cap_{j=1}^{2m} \cap_{i\in \mathcal I_j} E_j(i)$. By repeatedly applying the domain Markov property, we have
	\begin{align} \label{eq:irr-decomp}
		\phi_{B(2^{n_2})}^\xi (\mathcal A(R)) \le \phi_{B(2^{n_2})}^\xi (\cap_{j=1}^{2m} \cap_{i\in \mathcal I_j} E_j(i)) = \prod_{j=1}^{2m}\prod_{i\in \mathcal I_j} \E_{B(2^{n_2})}^\xi \left[ \phi_{D_{j,i}}^{\xi_{j,i}} (E_j(i)) \right],
	\end{align}
	where $D_{j,i} = \cup_{k=1}^{j-1} (\cup_{D \in \mathcal L_k} D) \cup (\cup_{k=1}^i \mathrm{Ann}_j(k))$, that is, $D_j$ is the union of all annuli up to level $j-1$ union the union of all annuli in on level-$j$ up to index $i$. The exception is $D_{2m} = B(2^{n_2})$. And $\xi_{j,i}$ is the random variable of boundary conditions on $\partial D_{j,i}$ induced by conditioning on the outside. $\xi_{2m} = \xi$. Note that $D_{j_1, i_1} \subset D_{j_2, i_2}$ if $j_1 < j_2$ or if $j_1 = j_2$ and $i_1<i_2$.
	
	Let $\tilde{E}_j(i)$ be the event that $E_j(i)$ occurs and the exit points of the crossings are separated, that is the distance between any two exit points are at least  $\delta/2m$ times the length of the boundary of the box they are on, for some $\delta > 1/8$. The following lemma is an arm separation statement that compares the separated event to the regular arm event.
	\begin{Lemma}\label{lemma:sep-fixed-scale}
		For any $i,j,\xi$, there is a $c_{12} = c_{12}(m)>0$ such that for any $j=1, \dots, 2m$,
		\begin{equation} \label{eq:sep-fxied-scale}
			\phi_{D_{j,i}}^\xi (E_j(i)) \le c_{12} \phi_{D_{j,i}}^\xi (\tilde{E}_j(i)).
		\end{equation}
	\end{Lemma}
	\begin{proof}
		This a classical result using RSW and FKG estimates except on half-annuli. Nonetheless, all parts of the classical argument apply. We refer the reader to the proof of \cite[Proposition 5.6]{CDH16}. 
	\end{proof}
	
	Applying Lemma \ref{lemma:sep-fixed-scale} to each probability in the RHS of \eqref{eq:irr-decomp} and then the domain Markov property and we have
	\begin{align*}
		\prod_{j=1}^{2m} \prod_{i\in \mathcal I_j} & \E_{B(2^{n_2})}^\xi \left[ \phi_{D_{j,i}}^{\xi_{j,i}} (E_j(i)) \right] \\
		&\le \tilde c_{12} \prod_{j=1}^{2m} \prod_{i\in \mathcal I_j} \E_{B(2^{n_2})}^\xi \left[\phi_{D_{j,i}}^{\xi_{j,i}} (\tilde E_j(i))\right] \\
		&= \tilde c_{12} \E_{B(2^{n_2})}^\xi \left[\phi_{D_{1,i_1}}^{\xi_{1,i_1}}(\tilde{E}_1(i_1))\right] \E_{B(2^{n_2})}^\xi \left[ \phi_{D_{1,i_2}}^{\xi_{1,i_2}} (\tilde E_1(i_2)) \right] \\
		&\qquad \quad \times \prod_{k=3}^{|\mathcal I_1|} \E_{B(2^{n_2})}^\xi \left[\phi_{D_{1,i_k}}^{\xi_{1,i_k}} (\tilde E_1(i_k))\right] \prod_{j=2}^{2m} \prod_{i\in \mathcal I_j} \E_{B(2^{n_2})}^\xi \left[\phi_{D_{j,i}}^{\xi_{j,i}} (\tilde E_j(i))\right] \\
		&= \tilde c_{12} \E_{B(2^{n_2})}^\xi \left[\phi_{D_{1, i_2}}^{\xi_{1, i_2}'} (\tilde{E}_1(i_1) \cap \tilde E_1(i_2))\right] \\
		&\qquad \quad \times \prod_{k=3}^{|\mathcal I_1|} \E_{B(2^{n_2})}^\xi \left[ \phi_{D_{1,i_k}}^{\xi_{1,i_k}} (\tilde E_1(i_k)) \right] \prod_{j=2}^{2m} \prod_{i\in \mathcal I_j} \E_{B(2^{n_2})}^\xi \left[ \phi_{D_{j,i}}^{\xi_{j,i}} (\tilde E_j(i))\right] \\
		& \cdots \\
		&= \tilde c_{12} \E_{B(2^{n_2})}^\xi \left[\phi_{B(2^{n_2})}^{\xi'} (\cap_{j=1}^{2m} \cap_{i\in \mathcal I_j}\tilde{E}_j(i)) \right],
	\end{align*}
	where $\xi'$ is a random variable of boundary conditions on $B(2^{n_2})$ and $\tilde c_{12}$ is some power of $c_{12}$. 

	It remains to ``glue'' the crossings that occur in the $\tilde{E}_j$ events together so that $\tilde{\mathcal A}(R)$ occurs, which we refer to Section \ref{sec:gluing} for details on the gluing constructions. We make a special note that connecting a crossing in the inner-most annulus inward to the boundary of $R$($\alpha_i$ or $\beta_i$) has a constant cost due to RSW. Then, there exists $c$ that depends only on $m$ such that for any arbitrary boundary condition $\xi'$,
	\begin{align*}
		\phi_{B(2^{n_2})}^{\xi'} (\cap_{j=1}^{2m} \tilde{E}_j) \le c \phi_{B(2^{n_2})}^{\xi'} (\tilde{\mathcal A}(R))
	\end{align*}  
	as desired.
	

	\begin{figure}
		\centering
		\includegraphics[width=0.67\textwidth]{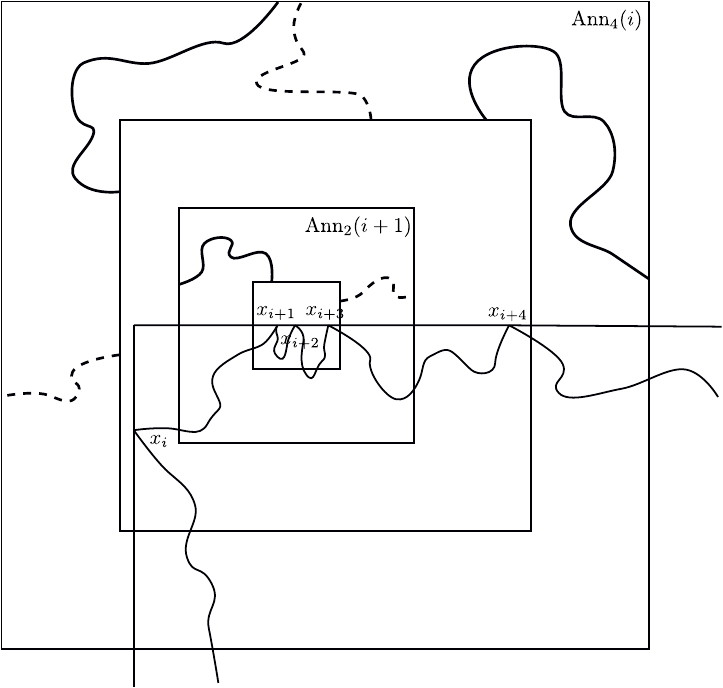}
		\caption{$x_i, x_{i+1}, x_{i+2}, x_{i+3}$, and $x_{i+4}$ are five vertices on $\partial B(2^{n_1})$. There are two disjoint crossings, one open and one dual-closed, in the annulus $\mathrm{Ann}_2(i+1)$. There are four disjoint crossings, alternatingly dual-closed and open, in the annulus $\mathrm{Ann}_4(i)$.}
		\label{fig:arm-sep-2}
	\end{figure}
	
	We now formally define $\mathcal I_j$ and $\mathcal L_j$. Recall that $x_1, \dots, x_{2m}$ are $2m$ vertices on $\partial B(2^{n_1})$. Let us again use cyclic indexing, i.e. $i= i \bmod (2m)$. Recall further that each level-$j$ annulus is crossed by $j$ arms and encloses $j+1$ vertices. The purpose of defining these annuli (and groupings of vertices) is to identify which arms are close relative to the scale and which ones are far away.
	
	We start with level $1$. Let $\mathcal I_1$ be the collection of indices such that the index $i$ is $I_1$ if the distance between $x_i$ and $x_{i+1}$ is logarithmically smaller than the distance between them and any other adjacent vertices, that is, $i \in \mathcal I_1$ if
	\begin{equation} \label{eq:level-1-cond}
		|x_i - x_{i+1}|< 2^{-5}\cdot \min \{|x_{i-1} - x_i|, |x_{i+1}- x_{i+2}|\}.
	\end{equation}
	For any $i \in \mathcal I_1$, we define 
	\begin{align*}
		\ell_1(i) := \min \{\ell: \exists x\in 2^\ell \Z^2 \cap \partial B(2^{n_1}) \text{ such that } B(x, 2^\ell) \supset \{x_i, x_{i+1}\}\}.
	\end{align*}
	Let $x_1(i)$ be the center for such a box $B(x_1(i), 2^{\ell_1(i)})$. If there are several choices for $x_1(i)$, we choose the first in lexicographical order. Next, we define 
	\begin{align*}
		\ell_1'(i) := \min\{\ell \ge \ell_1(i): B(x_1(i), 2^\ell) \ni x_{i-1} \text{ or } B(x_1(i), 2^\ell) \ni x_{i+2}\} - 3.
	\end{align*}
	Condition \eqref{eq:level-1-cond} guarantees the existence of $x_1(i)$ and ensures that $\ell_1(i) < \ell_1'(i) \le n_1$. Let $\mathrm{Ann}_{1}(i) := \mathrm{Ann}(x_1(i); 2^{\ell_1(i)}, 2^{\ell_1'(i)})$. Finally, we let $\mathcal L_1 := \{\mathrm{Ann}_1(i): i \in \mathcal I_1\}$. 
	
	For $j=2,\dots, 2m-2$, we define $\mathcal I_j$ and $\mathcal L_j$ inductively. Let $\mathcal I_j$ again be a collection of indices. An index $i$ is in $\mathcal I_j$ if
	\begin{equation} \label{eq:level-j-cond}
		\max_{k,\ell \in \{i, \dots, i+j\}} |x_k- x_\ell| < 2^{-5} \cdot \min \{|x_{i-1}-x_i|, |x_{i+j} - x_{i+j+1}|\}.
	\end{equation}
	For any $i \in \mathcal I_j$, we define
	\begin{align*}
		\ell_j (i) := \min &\{\ell: \exists x \in 2^\ell \Z^2 \cap \partial B(2^{n_1}) \text{ such that } \\
		& \quad B(x, 2^\ell) \supset \{x_i, \dots, x_{i+j}\} \cup (\cup_{h=1}^{j-1}\cup_{k=i}^{i+j} \mathrm{Ann}_h (k))\}
	\end{align*}
	where $\mathrm{Ann}_h(k) = \emptyset$ if $(k, k+1, \dots, k+h) \not \in \mathcal I_h$. Let $x_j(i)$ be the center for such a box $B(x_j(i), 2^{\ell_j(i)})$. If there are several choices for $x_j(i)$, we choose the first in lexicographical order. Next, we define
	\begin{align*}
		\ell_j'(i) := \min \{\ell \ge \ell_j(i): B(x_j(i), 2^\ell) \ni x_{i-1} \text{ or } B(x_j(i), 2^\ell) \ni x_{i+j+1}\} - 3.
	\end{align*}
	Again, condition \eqref{eq:level-j-cond} guarantees the existence of $x_j(i)$ and ensures that $\ell_j(i)< \ell_j'(i)\le n_1$. We let $\mathrm{Ann}_j (i) := \mathrm{Ann}(x_j(i); 2^{\ell_j(i)}, 2^{\ell_j'(i)})$ and $\mathcal L_j := \{\mathrm{Ann}_j(i): i \in \mathcal I_j\}$. Note that the definition of $\ell_j(i)$ ensures that level-$j$ annuli are disjoint from annuli of lower levels.
	
	The case $j= 2m-1$ is different from the previous cases: for there to be a level-$(2m-1)$ annulus, all $2m$ vertices must be concentrated relative to the scale of $\partial B(2^{n_1})$. We say $1 \in \mathcal I_{2m-1}$ if
	\begin{align*}
		\max_{k,\ell \in \{1,\dots, 2m\}} |x_k - x_\ell| < 2^{n_1-5}.
	\end{align*}
	If $\mathcal I_{2m-1}$ is nonempty, we define $\ell_{2m-1}$ similar to before:
	\begin{align*}
		\ell_{2m-1} := \min &\{\ell: \exists x \in 2^\ell \Z^2 \cap \partial B(2^{n_1}) \text{ such that } \\
		& \quad B(x, 2^\ell) \supset \{x_1, \dots, x_{2m}\} \cup (\cup_{h=1}^{2m-2}\cup_{k=1}^{2m} \mathrm{Ann}_h (k))\}
	\end{align*}
	Let $x_{2m-1}$ be the center for such a box $B(x_{2m-1}, 2^{\ell_{2m-1}})$. If there are several choices for $x_{2m-1}$, we choose the first in lexicographical order. For the second-to-last level, we define
	\begin{align*}
		\ell_{2m-1}' := n_1 - 2.
	\end{align*}
	Similarly to before, we define $\mathrm{Ann}_{2m-1}(1):= \mathrm{Ann}(x_{2m-1}; 2^{\ell_{2m-1}}, 2^{\ell_{2m-1}'})$. Let $\mathcal L_{2m-1} = \{\mathrm{Ann}_{2m-1}(1)\}$ if $\mathcal I_{2m-1}$ is nonempty and empty otherwise.
	
	Finally, we define $\mathcal L_{2m} := \{\mathrm{Ann}_{2m}(1)\} = \{\mathrm{Ann}(2^{n_1+1}, 2^{n_2})\}$.

\end{proof}

\begin{Remark}
	Although Lemmas \ref{lemma:ext-arm-sep} and \ref{lemma:int-arm-sep} are stated for $2m$ alternating arms, the proof can be adapted to accommodate any color sequence such that the dual-closed arms and defect edges are matched, thus including the three-arm case. For consecutive open arms, the constructions in step one and two remain the same except there are multiple open arms emenating from $\beta_i$ in the definition of $\mathcal A(R)$. This subsequently changes the definitions of $E_j(i)$, but the argument carries through as the arm separation statement still holds. Consecutive closed arms can be considered as having zero open arm between them, the argument for which follows the consecutive open arms case essentially.
\end{Remark}

\section{Outline of the Proof of Theorem \ref{thm:main}} \label{sec:outline}

As an extension to the results derived in \cite{DHS21}, the proof of the main result follows the same strategy with modifications in certain arguments. For completeness, we outline the proof with an emphasis on the present application and point to the main differences. An alternate, more detailed outline is offered in \cite[Section 2]{DHS21}.

The proof is essentially divided into three steps: In the first step, we construct shortcuts around edges on the lowest crossing and show that the existence of such shortcuts has a ``good'' probability. The second step uses an iterative scheme to improve upon shortcuts. Finally, we find the maximal collection of disjoint shortcuts and sum up the total savings.

\subsection*{Step 0: The lowest crossing $\ell_n$}
The estimate on the length of $\ell_n$ relies on the observation that $\ell_n$ consists only of three-arm points: since $\ell_n$ is the lowest crossing, by duality, from every edge $e$ in $\ell_n$ there are two disjoint open arms and a dual-closed arm to distance $\mathrm{dist}(e, B(n))$. In conjunction with some smoothness control, we have
\begin{equation}
	\E[\#\ell_n \mid \mathcal H_n] \le Cn^2 \pi_3(n).
\end{equation}

\subsection*{Step 1: Construction of shortcuts}
For any $\eps>0$, an edge $e$ on the lowest crossing $\ell_n$, we look for an arc $r$ over $e$ that saves at least $(1/\eps -1)\#r$ edges. The event $\hat E_k(e) = \hat E_k(e,\eps, \delta)$ discribes such an arc circumventing $e$ on scale $k$. The exact definition of $\hat E_k$ is quite involved, see \cite[Section 5]{DHS21}. We only state the properties and results relevant to the argument:
\begin{enumerate}
	\item $\hat E_k(e)$ depends only on $\mathrm{Ann}(e;2^k,2^K)$ where $K = k +  \lfloor \log (1/\eps) \rfloor$;
	
	
	\item \label{property:prop-5/4} For each $e \in \ell_n$, $\hat E_k(e,\eps, \delta)$ implies the existence of an $\delta$-shortcut around $e$. That is, there is an open arc $r\subset B(e, 3\cdot 2^k)$ such that $r$ only intersects with $\ell_n$ at its two endpoints $u(e)$ and $v(e)$ and
	\begin{equation*}
		\frac{\# r}{\#\tau} \le \delta;
	\end{equation*}
	where $\tau$ denotes the portion of $\ell_n$ between $u(e)$ and $v(e)$. See \cite[Proposition 5.4]{DHS21}.
	
	\item \label{property:5/29} $\hat E'_k$ is a similar event to $\hat E_k$ on scale $k$ that relates to a ``U-shaped region'' and $\mathfrak{s}_k$ is the shortest path in the U-shaped region.  If for some $\eps \in (0,\frac12)$, $\delta>0$, and $k\ge 1$, 
	\begin{equation*}
		\E[\# \mathfrak{s}_k \mid \hat E_k'] \le \delta 2^{2k} \pi_3(2^k)
	\end{equation*}
	holds,	then for all $L\ge 1$, 
	\begin{equation} \label{eq:ej-cond-3arm}
		\phi_{B(n)}^\xi(\hat E_k(e, \epsilon, \delta) \mid A_3(e, 2^L)) \ge c_0\eps^4 \quad \text{for all $L\ge 1$}.
	\end{equation}
	See \cite[Equation (5.29)]{DHS21}.
\end{enumerate}

Property (\ref{property:prop-5/4}) relies mostly on topological considerations and therefore applies to the random cluster model. For property (\ref{property:5/29}), we refrain from elaborating further despite the original proof using, at times, independence, generalized FKG, and gluing constructions. This is because we feel the techniques to convert these arguments for the random cluster model are sufficiently represented in the proofs that we do include, especially in that of the following proposition; and to completely reproduce all necessary parts of the proof of property (\ref{property:5/29}) would require lots of notation and mostly verbatim steps that translate directly for the random cluster model.

The key result in this step is that the probability that no shortcut exists for any scale $k$ is small:
\begin{Prop} \label{thm:no-shortcut}
	There is a constant $c_{13}$ such that if $\delta_j>0$, $j=1,\dots, L$, is a sequence of parameters such that for some $\eps \in (0,\frac14)$, 
	\begin{equation} \label{eq:delta-j}
		\E[\# \frak s_j \mid \hat E'_j] \le \delta_j 2^{2j} \pi_3(2^j),
	\end{equation}
	then for any $L'<L$,
	\begin{equation*}
		\phi_{B(n)}^\xi\left(\cap_{j=L'}^L \hat E_j(e,\eps,\delta_j)^c \mid A_3(e,2^L)\right) \le 2^{- c_{13}\frac{\eps^4}{\log(1/\epsilon)}(L-L')}.
	\end{equation*}
\end{Prop}




\begin{proof}[Proof of Proposition \ref{thm:no-shortcut} subject to Theorem \ref{thm:large-dev}]
	Let $E_k = \hat E_{kN}$. By property (\ref{property:5/29}), the combination of \eqref{eq:delta-j} with the observation that the occurrence of a circuit in $\mathrm{Ann}(2^{(10k+i)N}, 2^{(10k+i+1)N})$ conditional on a three-arm event has constant probability due to RSW and gluing constructions (see Proposition \ref{prop:gluing}) implies that
	\begin{align*}
		\phi_{B(n)}^\xi (E_{10k+5} \cap \hat{\mathfrak{C}}_k \mid A_3(2^{L})) \ge c_0' \eps^4
	\end{align*}
	for $0\le k\le \frac{L}{10N}-1$. Note that $c_{0}'$ is uniform in $k$. We observe the following chain of set inclusions with changes of indices including $j=\ell N$ in the equality and $\ell = 10k$ in the first inclusion:
	\begin{align*}
		\bigcap_{j=L'}^{L} \hat E_j^c = \bigcap_{\ell =\lceil \frac{L'}{N}\rceil}^{\lfloor \frac{L}{N}\rfloor} E_{\ell}^c
		\subset \bigcap_{k=\lceil \frac{L'}{10N}\rceil }^{\lfloor \frac{L}{10N}\rfloor-1} (E_{10k+5})^c
		\subset \left\{ \sum_{k=\lceil\frac{L'}{10N}\rceil}^{\lfloor\frac{L}{10N}\rfloor -1} \mathbf{1}\{E_{10k+5}, \hat{\mathfrak{C}}_k\} = 0 \right\}.
	\end{align*}
	Thus, applying Theorem \ref{thm:large-dev} by choosing $N= \lfloor \log(1/\eps)\rfloor$ and $L-L' \ge 40$, we obtain
	\begin{align*}
		\phi_{B(n)}^\xi \left(\cap_{k=L'}^L \hat E_{k}(e,\eps, \delta_j)^c \;\Big\vert\; A_3(e, 2^{L})\right) 
		& \le \phi_{B(n)}^\xi \left( \sum_{k=\lceil\frac{L'}{10N}\rceil}^{\lfloor\frac{L}{10N}\rfloor -1} \mathbf{1}\{E_{10k+5}, \hat{\mathfrak{C}}_k\} = 0 \:\Bigg\vert\; A_3(2^{L})\right) \\
		& \le \exp(-\frac{cc_0' \epsilon^4}{\log(1/\epsilon)}(L-L')). 
	\end{align*}
\end{proof}

\subsection*{Step 2: Iteration in the ``U-shaped region''}
In this step, we inductively improve the length of the ``best possible'' shortcuts for a fixed scale. The function of this step is to ensure that \eqref{eq:delta-j} is satisfied. \begin{Prop}[\cite{DHS21}, Proposition 7.1]
	There exist constants $C_1, C_2$ such that for any $\eps > 0$ sufficiently small, $L \ge 1$, and $2^k \ge (C_1\eps^{-4}(\log(1/\eps)^2))^L$, we have
	\begin{equation*}
		\E[\# \mathfrak{s}_k \mid \hat E'_k] \le (C_2 \eps^{1/2})^L 2^{2k}\pi_3(2^k).
	\end{equation*}
\end{Prop}
The constructions are detailed in \cite[Section 6 \& 7]{DHS21} and we only give the high level heuristics. In step $1$, shortcuts are constructed in a ``U-shaped'' region. Conditional on an event $\hat E_k'$ which is a superset of $\hat E_k$ for the ``U-shaped'' region at scale $2^k$, the starting estimate of a piece of shortcut is
\begin{equation} \label{eq:starting-shortcut}
	\E[\#\mathfrak s_k \mid E_k'] \le C_0 2^{2k}\pi_3(2^k).
\end{equation}
The factor $2^{2k}$ comes from the five-arm points in the construction. Suppose, at stage $i$, one can construct a shortcut of the order at most
\begin{equation*}
	\E[\#\mathfrak s_k \mid E_k'] \le \delta_k(i)2^{2k}\pi_3(2^k).
\end{equation*}
Through constructions detailed in \cite[Section 7]{DHS21}, we get an additional gain of $\sim \epsilon^{1/2}$ as long as there is enough space, i.e., when $2^k \ge C(\epsilon)^i$ for some $C(\epsilon)\sim \epsilon^{-4}(\log \frac1\epsilon)^2$. We then iterate this procedure. 

Since the proof of this proposition relies mostly on intricate algebraic manipulations, we simply cite the conclusion and refer the reader to the original paper for more explanation.

\subsection*{Step 3: Compilation}
The final estimate accounts for edges too close to the origin or the boundary, edges on $\ell_n$ that don't have shortcuts in step $1$, and a maximal collection of disjoint shortcuts that are optimized in step $2$. For the reader's benefit, we recreate the compilation here.

We first define a truncated box $\hat{B}(n) = B(n-n^\delta) \setminus B(n^\delta)$ for $\delta>0$ small enough such that $n^{1+2\delta} \le n^2\pi_3(n)$. For each $e\in \hat{B}(n)$, we let $L' = \lceil\frac{\delta}8 \log n\rceil$ and $L = \lfloor\frac{\delta}4 \log n\rfloor$.

We apply Proposition \ref{thm:no-shortcut} for $L' = \lceil\frac{\delta}8 \log n\rceil$ and $L = \lfloor\frac{\delta}4 \log n\rfloor$ and obtain
\begin{align*}
	\phi_{B(n)}^\xi(\text{there is no $n^{-c}$-shortcut around $e$} \mid e \in \ell_n)
	&\le \phi_{B(n)}^\xi\Big(\bigcap_{j = \lceil\frac{\delta}8 \log n\rceil}^{\lfloor\frac{\delta}4 \log n\rfloor} \hat{E}_j(e,\epsilon, n^{-c})^c \mid A_3(e, n^{\delta/2})\Big) \\
	&\le 2^{-c_{13} \epsilon^4\frac{\delta}{8} \log n } \\
	&\le n^{-\theta}
\end{align*}
for some $\theta>0$.

We choose a collection of $n^{-c}$-shortcuts around edges of $\ell_n$ such that the shortcuts are disjoint and the number of edges circumvented is maximal. Conditional on the existence of a horizontal crossing, any edge $e$ in $B(n)$ falls into one of three categories: in the margin of the box, with no $n^{-c}$-shortcut, or with a $n^{-c}$-shortcut. Thus, $S_n$ has the following estimate:
\begin{align*}
	\E_{\Z^2}[S_n \mid \mathcal H_n]
	&\le C n^{1+\delta} + n^{-\theta} \E[\# \ell_n\mid \mathcal H_n] +  n^{-c} \E[\# \ell_n \mid \mathcal H_n] \\
	&\le C n^{-\min\{\delta,\theta, c\}} n^2 \pi_3(n).
\end{align*}

\section{Estimating Without Reimer's Inequality} \label{sec:no-reimer}

In the radial case, there is no natural crossing like the lowest crossing to compare to. Instead, we consider ``lowest-like'' paths between successive circuits around the origin. One nuisance in this construction occurs when two circuits are close and there is not enough space for there to be three arms to a large distance. However, if this happens, closeby circuits form a bottleneck which implies an arm event with more than three arms. This ensures that the three-arm probability is an upper bound. The details of the construction are encapsulated in \cite[Lemma 2.3]{SR20}, and similarly \cite[Lemma 4.5, 4.7]{SR20}.

Recall that $A_3(n_1,n_2)$ denotes the three-arm event in the annulus $\mathrm{Ann}(n_1,n_2)$ and $\pi_3(n_1,n_2)$ its probability with domain $B(n)$ and boundary condition $\xi$. For any $j\ge 3$, let $A_j(n_1,n_2)$ ($\pi_{j}(n_1,n_2)$, resp.) denote the polychromatic $j$-arm event (probability, resp.) with exactly $j-1$ disjoint open arms and one dual-closed arm. Let $\pi'_j(n_1,n_2)$ denote the monochromatic $j$-arm probability. In an abuse of notation, for a box ``centered at an edge'', we write $B(e, n)$ in place of $B(e_x, n)$ where $e_x$ denotes the first endpoint of the edge $e$ in lexicographical order. 
\begin{Lemma}
	Fix $\eps>0$ and an integer $R$ such that for any $0\le n_1<n_2$, $\pi'_{2R+2}(n_1,n_2) \le \pi_3(n_1,n_2) (n_1/n_2)^\eps$. Let $\mathcal H_R(e,M)$ be the event that there exist $0=\ell_0\le \ell_1\le \cdots \le \ell_R = \lfloor \log(M)\rfloor$:
	\begin{enumerate}
		\item $A_{3,OOC}(e, 2^{\ell_1 - 1})$ occurs;
		\item for $i \ge 2$, if $\ell_{i-1} < \lfloor \log(M)\rfloor$, there are $2i$ disjoint open arms and one closed dual arm from $\partial B(e, 2^{\ell_{i-1}})$ to $\partial B(e, 2^{\ell_i - 1})$; and
		\item  if $\ell_R < \lfloor \log(M)\rfloor$, there are $2R + 2$ disjoint open arms from $\partial B(e, 2^{\ell_R})$ to $\partial B(e,M)$.
	\end{enumerate}
	Then,
	\begin{align*}
		\phi_{B(n)}^\xi (\mathcal H_R(e,M)) \le C M^{-r}\pi_3(M).
	\end{align*}
\end{Lemma}

The above estimate relies essentially on the following proposition.
\begin{Prop} 
Let $0\le n_1<n_2$, $i\ge 1$, 
	\begin{align*}
		\pi_{2i+1}(n_1,n_2) \le \left(\frac{n_1}{n_2}\right)^{(2i-2)\alpha} \pi_3(n_1,n_2).
	\end{align*}
for some $\alpha \in (0,1)$.
\end{Prop}
In Bernoulli percolation, this is done by applying Reimer's inequality. A weak form of Reimer's inequality for the random cluster model can be found in \cite{BG13}. However, it requires the events to not only have disjoint occurrences but also occur on disjoint clusters. The arms in the arm event $\pi_{2i+1}$ that \cite{SR20} concerns belong to the same cluster, since they are portions of consecutive circuits chained together by a radial arm. Therefore, the weak estimate is not applicable to our problem. We provide a proof using conditional probability and quad-crossing RSW.
\begin{proof}
	It suffices to show that
	\begin{align*}
		\pi_{2i+1}(n_1,n_2) \le \left(\frac{n_1}{n_2}\right)^{\alpha} \pi_{2i}(n_1,n_2)
	\end{align*}
	for $2i\ge 3$.
	
	Since there is at least one dual-closed arm in any configuration of $A_{2i+1}(n_1,n_2)$, we condition on a dual-closed arm and the first open arm on its clockwise side and the (consecutive) first $2i-2$ disjoint open arms on its counterclockwise side and apply the domain Markov property. As in the proof of Claim \ref{claim:reimer} (but omitting many details here), let $\mathcal U$ denote the random region that contains the $2i$ arms and whose boundaries consist of a dual-closed arm, an open arm, and portions of $\partial B(n_1)$ and $\partial B(n_2)$. Then, conditional on $\mathcal U$, 
	\begin{align}
		\phi_{B(n)}^\xi(A_{2i+1, CO\cdots O} (n_1,n_2)) 
		&= \sum_{\text{admissible } U} \phi_{B(n)}^\xi (A_{1,O}(n_1,n_2, U^c) \mid \mathcal U=U) \phi_{B(n)}^\xi(\mathcal U=U) \nonumber \\
		&= \sum_{\text{admissible } U} \E_{B(n)}^\xi \left[\phi_{B(n) \setminus U}^\eta (A_1(n_1, n_2, U^c))\right] \phi_{B(n)}^\xi(\mathcal U=U) \label{eq:cond-all-but-one} 
	\end{align}
	where $A_1(n_1,n_2, U^c)$ is the one arm event in $\mathrm{Ann}(n_1,n_2)$ restricted to $U^c$ and $\eta$ is uniquely determined by $\xi$ and $U$.
	By quad-crossing RSW estimates \eqref{eq:quad-rsw}, the one-arm probability decays at $(n_2/n_1)^{-\alpha}$ for some $\alpha \in (0,1)$. 
	Applying the above estimate into \eqref{eq:cond-all-but-one} and we have
	\begin{align*}
		\eqref{eq:cond-all-but-one} 
		\le \left(\frac{n_1}{n_2}\right)^\alpha \sum_{\text{admissible } U} \phi_{B(n)}^0(\mathcal U=U) 
		= \left(\frac{n_1}{n_2}\right)^\alpha \phi_{B(n)}^0 (A_{2i, CO\cdots O}(n_1,n_2)).
	\end{align*}
	We note that to apply quad-crossing RSW, the extremal distance for each quad $\mathrm{Ann}(2^{\ell}, 2^{\ell+1}) \setminus U$, which for convenience we call $\mathcal D$ here, needs to be uniformly lower bounded over all admissible $U$. To our advantage, bottlenecks in $\mathcal D$ make the extremal distance larger. The boundary of $\mathcal D$ defines four arcs: $(ab)$ on $\partial B(2^\ell)$, $(cd)$ on $\partial B(2^{\ell+1})$, and $(bc)$ and $(da)$ in the interior of $\mathrm{Ann}(2^\ell, 2^{\ell+1})$. Indeed, if $\mathcal D$ is contained in another quad $\mathcal D'$ with the same landing arcs as $\mathcal D$, then 
	\begin{align} \label{eq:ext-dist-comp}
		\ell_{\mathcal D}[(ab), (cd)] \ge \ell_{\mathcal D'}[(ab), (cd)].
	\end{align}
	We verify \eqref{eq:ext-dist-comp} in Appendix \ref{sec:ext-dist}. Let $\alpha$ be a topological path in $U$, disjoint from $(bc)$ and $(da)$, and $\mathcal D'$ be all of $\mathrm{Ann}(2^\ell, 2^{\ell+1})$ with arcs $(ab)$, $(bc)'$, $(cd)$, and $(da)'$, where $(bc)'$ consists of a portion of $\partial B(2^\ell)$, $\alpha$, and a portion of $\partial B(2^{\ell+1})$, see the blue arc in Figure \ref{fig:quad-crossing}, and similarly, $(da)'$ consists of another portion of $\partial B(2^\ell)$, $\alpha$, and another portion of $\partial B(2^{\ell+1})$, see the red arc in Figure \ref{fig:quad-crossing}. Clearly, $\mathcal D$ is contained in $\mathcal D'$. Then,
	\begin{align*}
		\ell_{\mathcal D'}[(ab), (cd)] \ge \frac{2^\ell}{\min\{\#(ab), \#(cd)\}} \ge \frac1{16}.
	\end{align*}

	\begin{figure} 
		\centering
		\includegraphics[width=0.4\textwidth]{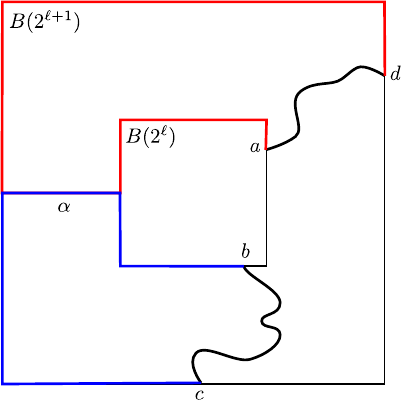}
		\caption{The blue arc is $(bc)'$ and the red arc is $(da)'$. They both traverse through $\alpha$.}
		\label{fig:quad-crossing}
	\end{figure}
\end{proof}

\appendix
\section{Extremal Distance and Resistance} \label{sec:ext-dist}

In this section, we verify \eqref{eq:ext-dist-comp} through the definition of extremal distance by the resistance of an electrical network.

\begin{Definition}[\cite{CDH16}]
	\begin{equation}
		\ell_\Omega[(ab),(cd)] := \sup_{g:\mathcal{E}(\Omega)\to \R_+} \frac{\left[\inf_{\gamma: (ab)\overset{\Omega}{\leftrightarrow} (cd)}\sum_{e\in \gamma} g_e\right]^2}{\sum_{e\in \mathcal{E}(\Omega)}g_e^2}.
	\end{equation}
\end{Definition}

Let $\Omega_2$ be a rectangle with vertices $a,b,c,d$, labeled in counterclockwise order. Then, the arcs $(ab)$, $(bc)$, $(cd)$, and $(da)$ are the four sides of the rectangle. Let $(bc)'$ be an arc from $b$ to $c$ contained in $\Omega_2$, and $(da)'$ be an arc from $d$ to $a$ contained in $\Omega_2$. Then, $\Omega_1$ bounded by $(ab)$, $(bc)'$, $(cd)$, and $(da)'$ is a subdomain of $\Omega_2$. We want to show 
\begin{equation} \label{eq:ext-dist-comp-2}
	\ell_{\Omega_1} [(ab), (cd)] \ge \ell_{\Omega_2} [(ab),(cd)].
\end{equation}

For any fixed $g:\mathcal E (\Omega_2) \to \R_+$, since $\mathcal E(\Omega_1)\subset \mathcal E (\Omega_2)$, we have $\{\gamma: (ab)\overset{\Omega_1}{\leftrightarrow}(cd)\} \subset \{\gamma: (ab)\overset{\Omega_2}{\leftrightarrow}(cd)\}$. Then,
\begin{align*}
	\inf_{\gamma: (ab)\overset{\Omega_1}{\leftrightarrow}(cd)} \sum_{e\in \gamma} g_e \ge \inf_{\gamma: (ab)\overset{\Omega_2}{\leftrightarrow}(cd)} \sum_{e\in \gamma} g_e.
\end{align*}
For the denominator, we use $\mathcal E(\Omega_1)\subset \mathcal E (\Omega_2)$ again:
\begin{align*}
	\sum_{e\in \mathcal E(\Omega_1)} g_e^2 \le \sum_{e\in \mathcal E(\Omega_2)} g_e^2.
\end{align*}
Therefore,
\begin{align*}
	\frac{\left[\inf_{\gamma: (ab)\overset{\Omega_1}{\leftrightarrow}(cd)} \sum_{e\in \gamma} g_e\right]^2}{\sum_{e\in \mathcal E(\Omega_1)} g_e^2} \ge \frac{\left[\inf_{\gamma: (ab)\overset{\Omega_2}{\leftrightarrow}(cd)} \sum_{e\in \gamma} g_e\right]^2}{\sum_{e\in \mathcal E(\Omega_2)} g_e^2}.
\end{align*}
\eqref{eq:ext-dist-comp-2} follows from taking supremum over all $g: \mathcal E(\Omega_2) \to \R_+$.

\bibliographystyle{plain}
\bibliography{reference}

\begin{thebibliography}{10}

\bibitem{AB99}
M.~Aizenman and A.~Burchard.
\newblock H{\"o}lder regularity and dimension bounds for random curves.
\newblock {\em Duke Math. J.}, 99:419--453, 1999.

\bibitem{AP96}
P.~Antal and A.~Pisztora.
\newblock On the chemical distance for supercritical {Bernoulli} percolation.
\newblock {\em Ann. Probab.}, 24:1036--1048, 1996.

\bibitem{BD12}
V.~Beffara and H.~Duminil-Copin.
\newblock The self-dual point of the two-dimensional random-cluster model is
  critical for $q\ge 1$.
\newblock {\em Probab. Theory Relat. Fields}, 153(3):511--542, 2012.

\bibitem{CDH16}
D.~Chelkak, H.~Duminil-Copin, and C.~Hongler.
\newblock Crossing probabilities in topological rectangles for the critical
  planar fk-ising model.
\newblock {\em Electron. J. Probab.}, 21:28 pp., 2016.

\bibitem{DHS17}
M.~Damron, J.~Hanson, and P.~Sosoe.
\newblock On the chemical distance in critical percolation.
\newblock {\em Electron. J. Probab.}, 22:1--43, 2017.

\bibitem{DHS21}
M.~Damron, J.~Hanson, and P.~Sosoe.
\newblock Strict inequality for the chemical distance exponent in
  two-dimensional critical percolation.
\newblock {\em Commun. Pure Appl. Math}, 74(4):679--743, 2021.

\bibitem{DS11}
M.~Damron and A.~Sapozhnikov.
\newblock Outlets of 2d invasion percolation and multiple-armed incipient
  infinite clusters.
\newblock {\em Probab. Theory Relat. Fields}, 150:257--294, 2011.

\bibitem{DC13}
H.~Duminil-Copin.
\newblock Parafermionic observables and their applications to planar
  statistical physics models.
\newblock {\em Ensaios Matematicos}, 25:1--371, 2013.

\bibitem{DC17}
H.~Duminil-Copin.
\newblock Lectures on the {Ising} and {Potts} models on the hypercubic lattice.
\newblock {\em arXiv e-prints}, July 2017.

\bibitem{DGHMT16}
H.~Duminil-Copin, M.~Gagnebin, M.~Harel, I.~Manolescu, and V.~Tassion.
\newblock Discontinuity of the phase transition for the planar random-cluster
  and {Potts} models with $q>4$.
\newblock {\em arXiv e-prints}, 2016.

\bibitem{DMT20}
H.~Duminil-Copin, I.~Manolescu, and V.~Tassion.
\newblock Planar random-cluster model: fractal properties of the critical
  phase.
\newblock {\em arXiv e-prints}, July 2020.

\bibitem{DST17}
H.~Duminil-Copin, V.~Sidoravicius, and V.~Tassion.
\newblock Continuity of the phase transition for planar random-cluster and
  {Potts} models with $1\le q\le 4$.
\newblock {\em Comm. Math. Phys.}, 349:47--107, 2017.

\bibitem{EK85}
B.~F. Edwards and A.~R. Kerstein.
\newblock Is there a lower critical dimension for chemical distance?
\newblock {\em J. Phys. A Math. Gen.}, 18(17):L1081--L1086, dec 1985.

\bibitem{ES88}
R.~G. Edwards and A.~D. Sokal.
\newblock Generalization of the {Fortuin-Kasteleyn-Swendsen-Wang}
  representation and {Monte Carlo} algorithm.
\newblock {\em Phys. Rev. D}, 38:2009--2012, Sep 1988.

\bibitem{FK72}
C.~M. Fortuin and P.~W. Kasteleyn.
\newblock On the random-cluster model: I. introduction and relation to other
  models.
\newblock {\em Physica}, 57(4):536--564, 1972.

\bibitem{Grassberger99}
P.~Grassberger.
\newblock Pair connectedness and shortest-path scaling in critical percolation.
\newblock {\em J. Phys. A Math. Gen.}, 32(35):6233--6238, aug 1999.

\bibitem{Grimmett06}
G.~Grimmett.
\newblock {\em The random-cluster model}, volume 333 of {\em Grundlehren der
  mathematischen Wissenschaften}.
\newblock Springer-Verlag, 2006.

\bibitem{GM90}
G.~R. Grimmett and J.~M. Marstrand.
\newblock The supercritical phase of percolation is well behaved.
\newblock {\em Proc. Roy. Soc. London Ser. A}, 430:439--457, 1990.

\bibitem{HN84}
S.~Havlin and R.~Nossal.
\newblock Topological properties of percolation clusters.
\newblock {\em J. Phys. A Math. Gen.}, 17(8):L427--L432, jun 1984.

\bibitem{HTWB85}
S.~Havlin, B.~Trus, G.~H. Weiss, and D.~Ben-Avraham.
\newblock The chemical distance distribution in percolation clusters.
\newblock {\em J. Phys. A Math. Gen.}, 18(5):L247--L249, apr 1985.

\bibitem{HHS84}
H.~J. Herrmann, D.~C. Hong, and H.~E. Stanley.
\newblock Backbone and elastic backbone of percolation clusters obtained by the
  new method of {\textquotesingle}burning{\textquotesingle}.
\newblock {\em J. Phys. A Math. Gen.}, 17(5):L261--L266, apr 1984.

\bibitem{HS88}
H.~J. Herrmann and H.~E. Stanley.
\newblock The fractal dimension of the minimum path in two- and
  three-dimensional percolation.
\newblock {\em J. Phys. A Math. Gen.}, 21(17):L829--L833, sep 1988.

\bibitem{Kesten87}
H.~Kesten.
\newblock Scaling relations for $2d$-percolation.
\newblock {\em Comm. Math. Phys.}, 109(1):109--156, 1987.

\bibitem{KZ93}
H.~Kesten and Y.~Zhang.
\newblock The tortuosity of occupied crossings of a box in critical
  percolation.
\newblock {\em Journal of Statistical Physics}, 70(3):599--611, Feb 1993.

\bibitem{MZ05}
G.~J. Morrow and Y.~Zhang.
\newblock The sizes of the pioneering, lowest crossing and pivotal sites in
  critical percolation on the triangular lattice.
\newblock {\em Ann. Appl. Probab.}, 15(3):1832--1886, 08 2005.

\bibitem{Nolin08}
P.~Nolin.
\newblock Near-critical percolation in two dimensions.
\newblock {\em Electron. J. Probab.}, 13:1562--1623, 2008.

\bibitem{RS20}
G.~Ray and Y.~Spinka.
\newblock A short proof of the discontinuity of phase transition in the planar
  random-cluster model with $q>4$.
\newblock {\em Comm. Math. Phys.}, 378:1977--1988, 2020.

\bibitem{ReevesSosoe20}
L.~Reeves and P.~Sosoe.
\newblock Equivalence of polychromatic arm probabilities on the square lattice.
\newblock {\em arXiv e-prints}, 2020.

\bibitem{SR20}
P.~Sosoe and L.~Reeves.
\newblock An estimate for the radial chemical distance in $2d$ critical
  percolation clusters.
\newblock {\em Stoch. Process. Their Appl.}, 147:145--174, 2022.

\bibitem{BG13}
J.~van~den Berg and A.~Gandolfi.
\newblock Bk-type inequalities and generalized random-cluster representations.
\newblock {\em Probab. Theory Relat. Fields}, 157:157--181, 2013.

\end{thebibliography}

\end{document}